\documentclass{amsart}

\usepackage{amsmath,amssymb,amsthm,a4wide}
\usepackage{subcaption}

\usepackage{graphicx, tikz,url,xcolor}
\usetikzlibrary{decorations.fractals}

\newtheorem{theorem}{Theorem}
\newtheorem*{thm}{Theorem}

\newtheorem{lemma}{Lemma}

\usepackage{xcolor}

\theoremstyle{definition}

\newtheorem{assumption}{Assumption}

\newtheorem{cor}{Corollary}
\newtheorem{defn}{Definition}

\theoremstyle{remark}
\newtheorem{remark}[defn]{Remark}

\DeclareMathOperator{\area}{\text{area}}

\DeclareMathOperator{\tTop}{p}
\DeclareMathOperator{\tBott}{q}
\DeclareMathOperator{\epsB}{\epsilon_b}
\DeclareMathOperator{\epsT}{\epsilon_t}

\DeclareMathOperator{\tTCent}{\left( p\, --\, \frac{1}{2} \right)}
\DeclareMathOperator{\tBCent}{\left( q\, --\, \frac{1}{2} \right)}
\DeclareMathOperator{\Th}{\theta}

\DeclareMathOperator{\sinc}{\text{sinc}}

\newcommand{\pmat}[1]{\begin{pmatrix} #1 \end{pmatrix}}

\begin{document}

\title[]{On the behavior of $1$-Laplacian Ratio Cuts on nearly rectangular domains}
\subjclass[2010]{} 

\author[]{Wesley Hamilton}
\address{Department of Mathematics, University of North Carolina at Chapel Hill, NC 27599, USA}
\email{wham@live.unc.edu}

\author[]{Jeremy L. Marzuola}
\address{Department of Mathematics, University of North Carolina at Chapel Hill, NC 27599, USA}
\email{marzuola@email.unc.edu}

\author[]{Hau-tieng Wu}
\address{Department of Mathematics and Department of Statistical Science, Duke University, Durham, NC 27708, USA}
\email{hauwu@math.duke.edu}

\begin{abstract}  Given a connected set $\Omega_0 \subset \mathbb{R}^2$, define a sequence of sets $(\Omega_n)_{n=0}^{\infty}$
where $\Omega_{n+1}$ is the subset of $\Omega_n$ where the first eigenfunction of the (properly normalized) Neumann $p-$Laplacian
$  -\Delta^{(p)} \phi = \lambda_1 |\phi|^{p-2} \phi$ is positive (or negative).  For $p=1$, this is also referred to as the Ratio Cut of the domain. We conjecture that, unless $\Omega_0$ is
an isosceles right triangle, these sets converge to the set of rectangles with eccentricity bounded by 2 in the Gromov-Hausdorff distance as long as they have a certain
distance to the boundary $\partial \Omega_0$. We establish some
aspects of this conjecture for $p=1$ where we prove that (1) the 1-Laplacian spectral cut of domains sufficiently close to rectangles is a circular arc that is closer to flat than the original domain (leading eventually to 
quadrilaterals) and (2) quadrilaterals close to a rectangle of aspect ratio $2$ stay close to quadrilaterals and move closer to rectangles in a suitable metric.
We also discuss some numerical aspects and pose many open questions.
\end{abstract}

\maketitle

\section{Introduction and Motivation}

 This paper is motivated by work of Szlam, Maggioni, Coifman \& Bremer \cite{szlam2} and an observation made explicit by
Szlam \cite{szlam}: taking iterated spectral cuts induced by the nodal set of the first non-constant eigenfunction for the Neumann $p$-Laplacian seems to converge to rectangles in shape. 
It is already observed in \cite{szlam} that starting with an isosceles right triangle will lead to a spectral cut along the symmetry axis and produce two smaller isoceles right triangles: no convergence to rectangles takes place. However, this is unstable under small perturbations of the initial domain. 
\begin{center}
\begin{figure}[ht!]
\includegraphics[width=0.85\textwidth]{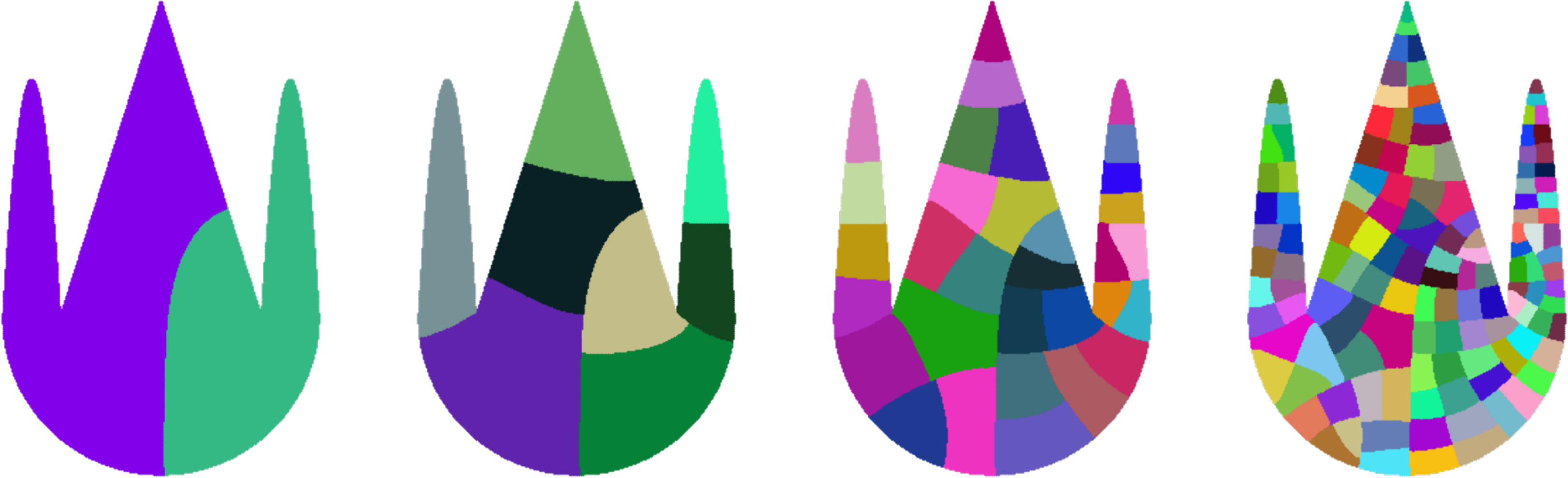}
\caption{Iterated spectral cuts of the standard graph Laplacian seem to lead to rectangles in a generic, non-convex, non-smooth domain.}
\end{figure}
\end{center}
\vspace{-20pt}

We believe this to be a fascinating question in itself but an affirmative answer would also be useful in guaranteeing that iterative spectral partitioning is an effective method to partition domains and, ultimately, graphs and data (see Irion \& Saito \cite{saito} where this phenomenon is exploited).  More importantly, a better understanding of this problem will shed light on the more general data analysis algorithms based on the $p$-Laplacian. More precisely, let $\Omega \subset \mathbb{R}^2$ be an open, bounded, and connected domain. We now propose an iterative subdivision of $\Omega$ as follows.
For any $1 \leq p < \infty$, the ground state of the $p-$Laplacian can be written as
\begin{equation} \label{definition p Lap ground state}
\lambda_{1,p}(\Omega) = \inf_{\int_{\Omega}{f dx} = 0} \frac{\int_{\Omega}{|\nabla f|^p dx}}{ \int_{\Omega}{|f|^p dx}}.\end{equation}
It is known that the function $f$ minimizing this functional exists and we can use it to iteratively define
\begin{equation}
\label{speccut}
 \Omega_{n+1} = \left\{ x \in \Omega_n: f(x) \geq 0\right\}\,,
 \end{equation}
where $n=0,1,2,\ldots$ and $\Omega_0:=\Omega$. 
The function $f$ is only defined up to sign, so restricting to the part of the domain where it is positive is without loss of generality. We raise the following conjecture (and refer to the subsequent paragraphs for clarification and obvious obstacles).
\begin{quote} \textbf{Main Conjecture.} If $\Omega_0$ is not the isosceles right triangle (having angles 45-90-45), then the sequence of sets $(\Omega_n)_{n=1}^{\infty}$ converges to the set of rectangles with eccentricity
bounded by 2 in the Gromov-Hausdorff distance.
\end{quote}
It is clear that one cannot expect convergence to a fixed rectangle: in general, the spectral cut of an $a \times b$ rectangle with $a > b$ will be given by two $(a/2) \times b$ rectangles and, as long as $a \leq 2b$,
the next cut would then yield two $(a/2) \times (b/2)$ rectangles. This motivates a refined question.

\begin{quote} \textbf{Question.} Do $(\Omega_{2n})_{n=1}^{\infty}$ and $(\Omega_{2n+1})_{n=1}^{\infty}$ converge in shape to a fixed rectangle?
\end{quote}

There is one obvious obstruction: if $\Omega_{0}$ is not already a rectangle, then while performing iterated subdivisions, there is always a sequence of choices for the sign of the
eigenfunction that ensures that part of the boundary of $\Omega_{n}$ coincides with part of the boundary of $\Omega$ which could possibly be quite ill-behaved (say, fractal).
However, for a sequence of choices of signs that leads to domains bounded away from $\partial \Omega_0$ ('deep' inside the domain) this should not be the case.

\begin{figure} 
\centering
\includegraphics[width=.4\textwidth]{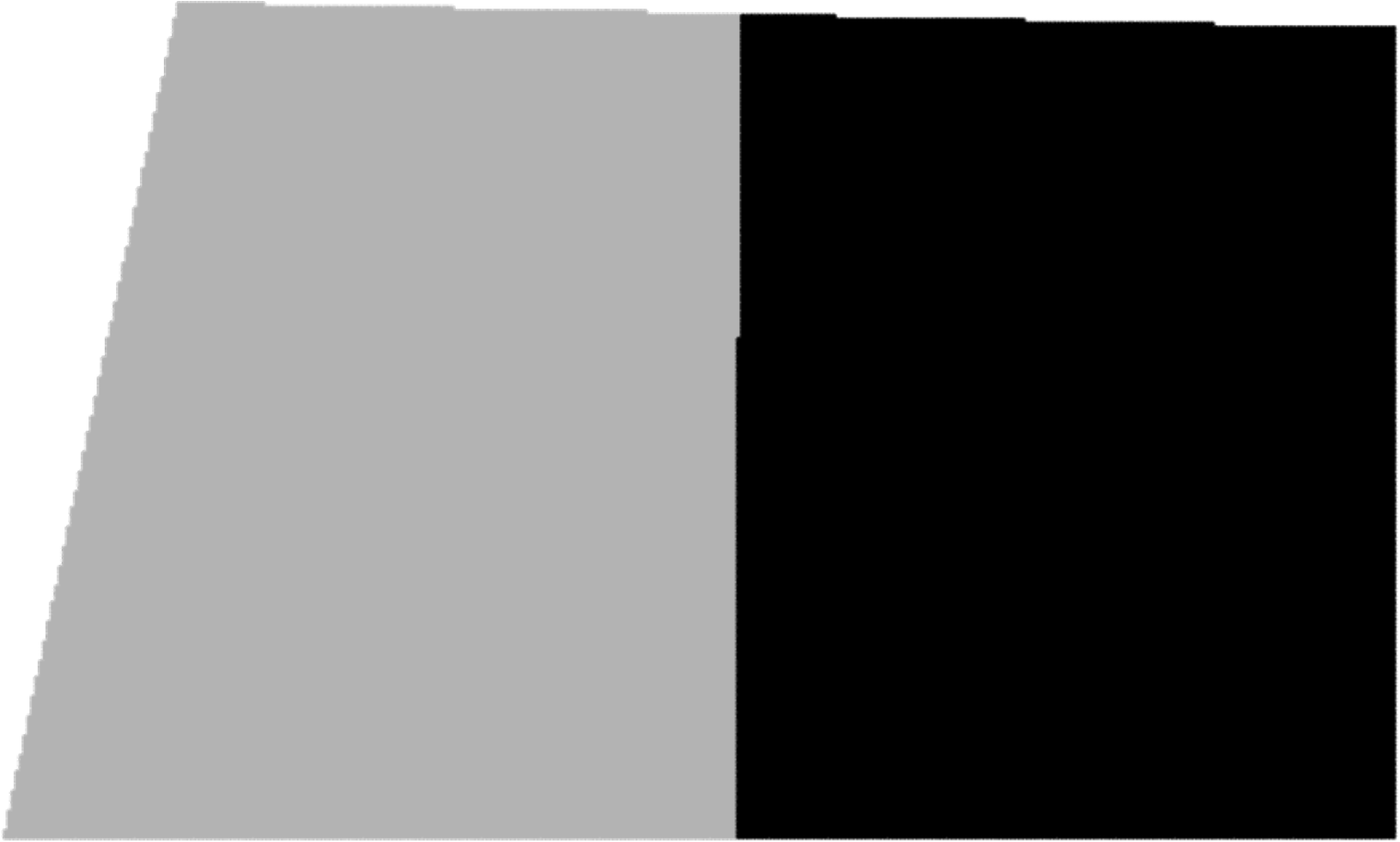}\hspace{60pt}
\includegraphics[width=.4\textwidth]{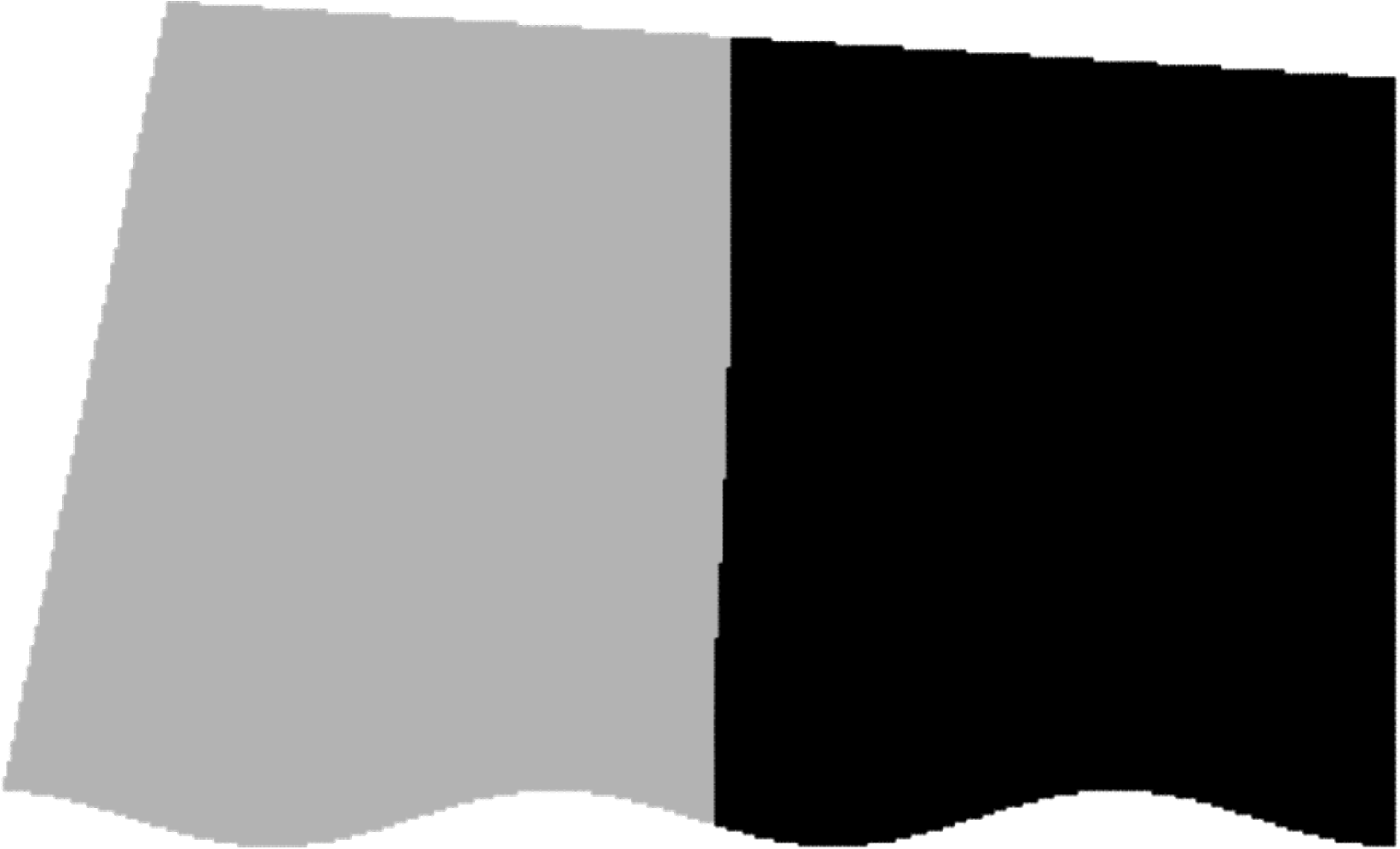}
\caption{Left: the spectral cut of a quadrilateral determined by four corners, $(0, 0)$, $(\pi/25, 3/5)$, $(1,0)$, and $(1, 3/5-\exp(1)/100)$ provided by the graph 1-Laplacian. Right: the spectral cut of a shape close to a quadrilateral provided by the graph 1-Laplacian.}\label{Fig:One}
\end{figure}

While this particular question seems to be novel, the problem of trying to understand the geometry of nodal cuts induced by the $p-$Laplacian or general nonlinear operators has
been studied for a long time. We refer to \cite{plap,plap2,plap3,plap4,plap5,plap6,plap7,plap8,plap9,plap0,plap10,plap11} and references therein. We especially
emphasize the works of Gajewski \& G\"artner \cite{plap7, plap8} who study the behavior of the cut as $p \rightarrow 1$ as a means of finding effective ways of separating the
domain into two roughly equally sized domains, as well as the work of Parini \cite{plap11} studying the limit of the cut as $p \rightarrow 1$ under Dirichlet boundary conditions.
Many of these results are posed for Dirichlet conditions where the effective functional as $p \rightarrow 1$ is given by
$$ \inf_{E \subset \Omega}{  \frac{ \mathcal{H}^{n-1}(\partial E)}{\mathcal{H}^{n}(E)}} \quad \mbox{while the Neumann case induces} \quad \inf_{E \subset \Omega}{  \frac{ \mathcal{H}^{n-1}(\partial E)}{\mathcal{H}^{n}(E)\mathcal{H}^{n}(\Omega \setminus E)}}.$$
When the domain has Neumann boundary conditions, the quantity above is referred to as the {\it Ratio Cut}.
New phenomena arise as a consequence. One interesting problem, that may also be of interest in the Dirichlet case, is the stability of the nodal set of the $p-$Laplacian
as a function of $p$. 

In the result presented here, we will focus on the $1$-Laplacian, where we can establish some preliminary results that suggest stability of rectangular domains as attractors for the proposed spectral dynamical system on domains.  In particular, from henceforward, we refer to the {\it $1$-spectral cut iteration} of a domain as the operation defined in \eqref{speccut} with $p=1$ and the {\it $1$-spectral cut} as the set $C_{n} := \left\{ x \in \Omega_n: f(x) = 0\right\}$. To simplify the terminology, we use {\it spectral cut} and {\it $1$-spectral cut} interchangeably.  %\jlm{Inserted a figure that shows numerical Ratio Cuts ... include a Remark about the methods/shortcomings?}

\textbf{Numerics.} For the reader's convenience, we briefly recall here the numerical implementation of the p-Laplacian and its related spectral cut \cite{Buhler_Hein:2009,Hein_Buhler:2010}.
Take a point cloud $\mathcal{X}:=\{x_i\}_{i=1}^N\subset (\mathcal M,d)$, a metric space $\mathcal M$ with the metric $d$. Construct an affinity matrix ${\bf W}\in \mathbb{R}^{N\times N}$, where $\bf W_{ij}$ is the affinity between $x_i$ and $x_j$. It is associated with an undirected affinity graph with $N$ vertices, where the affinity between $x_i$ and $x_j$, $w_{ij}$, is ${\bf W}_{ij}$ for $i,j=1,\ldots,N$. The {\em graph $p$-Laplacian} is defined by
\begin{equation}
\Delta_p f(i)=\sum_{j=1}^N w_{ij}\phi_p(f(i)-f(j)),
\end{equation}
where $f\in\mathbb{R}^N$ is a function defined on the vertices and $\phi_p(x)=\texttt{sign}(x)|x|^{p-1}$ for $x\in \mathbb{R}$. When $p=2$, this gives the bilinear form defined by the standard graph Laplacian ${\bf D} - {\bf W}$, where ${\bf D}=\texttt{diag}({\bf W}{\bf 1})$ and ${\bf 1}$ is a $N$-dim vector with all entries $1$. Clearly, in general the graph p-Laplacian is nonlinear. We have
\begin{equation}
\langle f,  \Delta_p f \rangle = \frac12 \sum_{i,j=1}^N w_{ij} |f_i - f_j|^p
\end{equation}
for $1 \leq p < \infty$.  See for instance \cite{Buhler_Hein:2009} for a discussion of the relationship between the variational formulation and the discrete operator formulation.

A real number $\lambda$ is called an  eigenvalue for the graph $p$-Laplacian if there exists a non-zero vector $f\in\mathbb{R}^N$ so that \cite[Definition 3.1]{Buhler_Hein:2009}
\begin{equation}
(\Delta_p f)_i=\lambda \phi_p(f_i)\,,
\end{equation}
where $i=1,\ldots,N$. We call $f$ the $p$-eigenfunction of of the graph $p$-Laplacian associated with the eigenvalue $\lambda$. We know that ${\bf 1}$ is the trivial eigenvector with eigenvalue $0$, and we have \cite[Lemma 3.2]{Buhler_Hein:2009}
\begin{equation}
{\bf 1}^T\phi_p(f)=0
\end{equation} 
if the eigenvector $f$ is non-trivial.
It is shown in \cite{Buhler_Hein:2009} that a non-zero function $f$ is an eigenvector if and only if it is a critical point (local minima) of the functional 
\begin{equation}
F_p(f)=\frac{\langle f,  \Delta_p f \rangle}{\|f\|_{\ell^p}^p}\,.
\end{equation}
Note that $F_p(f)$ is the discretization of the functional shown in \eqref{definition p Lap ground state}.
In this work, our main interest is the second $p$-eigenvector of the graph $p$-Laplacian for the spectral clustering purpose. In \cite{Buhler_Hein:2009}, the second eigenvalue is shown to be the global minimum of the functional
\begin{equation}
F^{(2)}_p(f):=\frac{\langle f,  \Delta_p f \rangle}{\texttt{var}_p(f)},
\end{equation}
where 
\begin{equation}
\texttt{var}_p(f)=\min_{c\in\mathbb{R}}\sum_{i=1}^N |f_i-c|^p\,,
\end{equation} 
and the corresponding eigenvector is then given by 
\begin{equation}
f_p^{(2)}=f^*-c^*{\bf 1}, 
\end{equation}
where $f^*$ is any global minimizer of $F^{(2)}_p$ and $c^*=\arg\min_{c\in \mathbb{R}} \sum_{i=1}^N |f^*_i-c|^p$. 
Like the usual spectral clustering, once we have $f_p^{(2)}$, we cluster the point cloud by the signs of its entries.  
For $p=1$, 
the consistency of spectral cut with the graph $1$-Laplacian was studied in \cite{trillos2016consistency}.
Numerically, we apply the nonlinear inverse power method proposed in \cite{Hein_Buhler:2010} to evaluate the iterative bi-partition of a given 2-dimensional domain.
In Figure \ref{Fig:One}, we present some numerically computed Ratio Cuts for nearly rectangular domains.

\textbf{Outline.} The paper proceeds as follows.  In Section \ref{sec:results}, we highlight the two main theorems we can prove on stability of near rectangular domains.  We also present some open problems to be considered naturally as generalizations of these theorems.   In Section \ref{sec:pf1}, we give the full proof that the spectral cut algorithm converges to a rectangle with bounded aspect ratio if the initial domain is near a rectangle in Gromov-Hausdorff sense as will be carefully laid out below.  Section \ref{sec:pf3} provides the details for the proof that quadrilaterals near the rectangle of aspect ratio $2$ in terms of small angle deviations from $90$ degrees will converge under the spectral cut algorithm to rectangles with bounded aspect ratio.   In the appendix we gather some long calculations that are useful in analyzing the Ratio Cut in a neighborhood of a quadrilateral.

 %In Section \ref{app}, we show some numerical simulations enacting the spectral cut dynamical system for two trapezoidal domains that are nearly rectangular.  All the simulations give strong evidence for the truth of the Main Conjecture stated above upon enough refinement.  The numerical schemes we use for the $1$-Laplacian are built around the $1$-Graph Laplacian constructed from a large uniform point set distributed throughout the domain. 

\section*{Acknowledgements} 
We thankfully acknowledge the generous support of NSF CAREER Grant DMS-1352353  (J.L. Marzuola \& W. Hamilton).  We thank Stefan Steinerberger for starting this project with us and for many helpful conversations along the way.

\section{results} 
\label{sec:results}
We prove two results that, while not establishing the main conjecture, do seem to suggest a mechanism by which this procedure happens.
$$ \mbox{certain shapes} \underbrace{\implies}_{\mbox{Theorem 1}} \mbox{nearly straight cuts} \implies \mbox{curved quadrilaterals} \underbrace{\implies}_{\mbox{Theorem 2}} \mbox{rectangles.}$$
The missing steps are as follows: (1) we do not know whether a generic $\Omega_0 \subset \mathbb{R}^2$ will ever produce domains $\Omega_n$ for which Theorem 1 becomes applicable, for example, when $\Omega_0$ has a fractal boundary; 
and (2) we do not know whether the dynamical system on the space of rectangles ever produces a quadrilateral sufficiently close to the set of rectangles for Theorem 2 to become applicable.

\subsection{Rectangular stability.} The first of the two main results states that the spectral cut of domains that 'roughly' look like rectangles are being cut
in the middle. More formally, let us carefully describe the domains we will consider.
\begin{assumption}
\label{rectangle_assumptions}
We will work with domains that are small perturbations of a rectangle in the following sense:
\begin{itemize} 
\item The domain $Q$ is a perturbed
rectangle that is close in the Gromov-Hausdorff distance to a reference rectangle $R$: 
$$d_{GH}(Q, R) \leq \varepsilon \left( \mbox{length of the shorter side of}~R\right)\,,$$
where $\varepsilon>0$ is sufficiently small.
\item
In a roughly $10\sqrt{\varepsilon}-$neighborhood of the two intersection points of the $1$-spectral cut of $R$ with the boundary, the boundary of $Q$ 
can be written as graphs of functions of the associated boundary segments of $R$ (see Fig. \ref{fig:cut}). Moreover, each of these functions can be well approximated by a parabola with bounded, small curvature and small Lipschitz constant.
\end{itemize}
\end{assumption}

\begin{center}
\begin{figure}[ht!] % \label{fig:complete}
\begin{tikzpicture}[scale = 2]
\draw [] (0,0) -- (1.61,0) -- (1.61, 1) -- (0, 1) -- (0,0);
\draw [dashed] (0.8,-0.2) -- (0.8, 1.2);
\draw [ultra thick] (0,0) to[out=10,in=170] (1, 0);
\draw [ultra thick] (1,0) to[out=340,in=200] (1.61, 0);
\draw [ultra thick] (1.61,0) to[out=80,in=280] (1.61, 1);
\draw [ultra thick] (1.61,1) to[out=170,in=10] (0, 1);
\draw [ultra thick] (0,1) to[out=260,in=100] (0, 0);
\draw [ultra thick, dashed] (0.8, -0.2) to[out=85,in=277] (0.81, 1.2);
\end{tikzpicture}
\caption{A perturbed rectangle $Q$ close to a reference rectangle $R$: the longer part of the boundary of $Q$ can be written as the graph of a Lipschitz function. The goal is to show that the spectral cut of $Q$ is close to that of $R$.}
\label{fig:cut}
\end{figure}
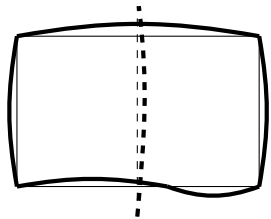
\end{center}

\begin{theorem}
\label{thm1} 
For a domain satisfying Assumption \ref{rectangle_assumptions}, the spectral cut occurs in a $\sqrt{\varepsilon}-$neighborhood of the midpoint of the long axis.  Moreover, there is a function $c(\varepsilon)$ tending to
0 as $\varepsilon \rightarrow 0$ such that, if said part of the boundary can be written as a Lipschitz function with Lipschitz constant $L \leq \sqrt{2} - c(\varepsilon)$ in the $\sqrt{\varepsilon}-$neighborhood of the spectral cut of $R$, then the spectral cut of $Q$
is a circular arc near a straight line.
\end{theorem}

Numerical examples suggest that this result starts becoming applicable rather quickly: already a small number of cuts seems to suffice to produce roughly
rectangular shapes. As soon as Theorem 1 becomes applicable the spectral cuts will start being closer and closer to straight lines, which immediately implies that many shapes will end
up approaching quadrilaterals after several additional steps. The next result shows that as soon as we are dealing with quadrilaterals close to a rectangle, the procedure
is smoothing and produces quadrilaterals closer to the rectangle; this has to be understood in the usual sense of 'cuts' away from the boundary: a quadrilateral
having a $91^{\circ}$ degree angle will always pass this angle on to one of its descendants.  
We note a certain similarity of Theorem \ref{thm1} to the result of Grieser-Jerison \cite{GrieserJerison} for the standard Laplacian on rectangular domains of high aspect ratio with one side described by a curve.  See also the recent work \cite{BCM} where general curvilinear quadrilaterals were considered.  

\subsection{Near Curvilinear Quadrilaterals} We settle our conjecture in the case of $\Omega$ being near a quadrilateral that is close to a rectangle: we prove convergence to a rectangular shape in the Gromov-Hausdorff distance 'away from the boundary' (in the sense
discussed above: rectangles incorporating parts of the initial boundary or boundaries created by very early initial spectral cuts need never be regular).
We introduce a qualitative way of measuring distance to quadrilaterals as follows: given a curvilinear quadrilateral $Q$ with angles $\alpha, \beta, \gamma, \delta$ and sides described by the curves $y = \gamma_1 (x)$, $y = \gamma_2(x)$, $x = \gamma_3(y)$ and $x  = \gamma_4 (y)$, we
define the functional
\begin{equation}\label{Definition I(Q)}
I(Q) =  \left| \alpha- \frac{\pi}{2} \right| + \left| \beta- \frac{\pi}{2} \right| + \left| \gamma- \frac{\pi}{2} \right| + \left| \delta- \frac{\pi}{2} \right| + \max_{1 \leq j \leq 4} \sup_s (|\gamma_j' |(s)   + |\gamma_j'' |(s)+ |\gamma_j''' |(s)) .
\end{equation}
We particularly want our domains to be well approximated by parabolic curves on the sides intersecting the ratio cut, and will refer to such domains as {\it approximately parabolic curvilinear quadrilaterals}. 

\begin{remark}
This is actually quite a bit more restrictive than one needs, though is certainly sufficient; otherwise, the theorem becomes harder to frame as one must introduce a large number of cases for behaviors of each boundary component of the domain.  However, in the calculations below, we will work with domains that are perturbations of a base rectangle with a reasonable aspect ratio, and whose top and bottom boundary components can be well-approximated by parabolas in a neighborhood near the axis of symmetry of the base rectangle.  As a result, we can dramatically change regularity requirements for the left and right curves in our model calculations as long as the perturbations are nearly symmetric in area and bounded by a sufficiently small constant.  
\end{remark}

\begin{theorem} 
\label{thm2}
For an approximately parabolic curvilinear quadrilateral of aspect ratio near $1:2$, there exists $\varepsilon_1 > 0$ such that if $I(Q) \leq \varepsilon_1$, then the $1$-spectral cut induced
by the $L^1-$Laplacian is a circular arc with opening angle bounded above by $\varepsilon_1/8$.  
\end{theorem}
The proof also shows that the constant $1/8$ cannot be further improved.
This result implies that near-quadrilateral regions have $1$-spectral cuts that are closer to a quadrilateral in Gromov-Hausdorff distance. This implies exponentially fast convergence to rectangles in shape. The way the result is obtained actually allows for a fairly precise understanding of what happens (in particular, it can be used to show that there is a choice of signs such that $I(Q_n)$ grows substantially along the subsequence). We refer to the proof for details.

\subsection{Open problems} These results naturally suggest several open problems; we only name a few that seem 
particularly natural.

\begin{enumerate}

\item Prove the main conjecture for $p=1$; that is, show that a generic $\Omega_0 \subset \mathbb{R}^2$ will produce domains $\Omega_n$ for which Theorem 1 can be applied, and show that the dynamical system on the space of rectangles produces a quadrilateral sufficiently close to the set of rectangles for Theorem 2 to become applicable. Is it possible to transfer some of the arguments to the range $p \in (1,1+\varepsilon_0)$?  What happens
as $p \rightarrow \infty$? 
\item What happens in higher dimensions, or even to domains with curvature, like manifolds? One would still expect the cut to have a smoothing effect but the types of geometric obstruction that
one could encounter in the process may be more complicated. Is the generic limit given by a rectangular box? If not, is there a finite set of shapes
that can arise in the limit?
\item It seems natural to expect that similar results should hold true under Dirichlet conditions. This would require the study of the variational problem 
$$ \inf_{E \subset \Omega}{  \frac{ \mathcal{H}^{n-1}(\partial E)}{\mathcal{H}^{n}(E)}}.$$
Can the results be extended to this case?
\item Experimentally, we observe that the nodal line is (generically) fairly stable under small perturbations of the value of $p$. Can any result in this
direction be made precise? Does it help to assume  the domain $\Omega$ to be convex?
\end{enumerate}

\section{proof of theorem 1}
\label{sec:pf1}

\subsection{Preliminaries.} The crucial ingredient in our approach is that the $p$-Laplacian degenerates as $p \rightarrow 1^+$. The minimum of the associated energy functional characterizing the ground state is not assumed by any continuous
function: reinterpreting the functional in terms of total variation, the extremal function is constant on two sets in the domain that are separated by
an $(n-1)-$dimensional hypersurface. More precisely, we have the following consequence of the coarea formula (see \cite{stein} for further implications
of this result).

\begin{thm}[Cianchi, \cite{Cianchi}] Let $\Omega \subset \mathbb{R}^{n}$ be open, bounded and connected. Then we have the sharp Poincar\'e-type inequality for any sufficiently smooth functions $u$:
$$\left\|u-\frac{1}{|\Omega|}\int_{\Omega}{u(z)dz}\right\|_{L^{1}(\Omega)} \leq
 \left(\sup_{\Gamma \subset \Omega}{\frac{2}{\mathcal{H}^{n-1}(\Gamma)}\frac{|S||\Omega \setminus S|}{|\Omega|}}\right)\left\|\nabla u\right\|_{L^{1}(\Omega)},$$
where $\Gamma \subset \Omega$ ranges over all surfaces which divide $\Omega$ into two connected open subsets $S$ and $\Omega \setminus \overline S$. 
\end{thm}

This means that the nodal line $\Gamma$ defining the $1$-spectral cut is a hypersurface partitioning $\Omega$ into two sets $S, \Omega \setminus \overline S$
$$ \mbox{so as to minimize the quantity} \qquad \frac{\mathcal{H}^1(\Gamma) |\Omega|}{|S| |\Omega \setminus S|}.$$
It is relatively easy to check that this value is assumed if we relax the $L^1$-norm of the gradient and re-interpret it as total variation. Let us assume that $f(x) = a \chi_{S} + b \chi_{\Omega \setminus S}$.
We want $f$ to have mean value 0, which leads to
$$ a |S| + b |\Omega \setminus S| = 0 \quad \mbox{and thus} \quad b = -a\frac{|S|}{|\Omega \setminus S|} \quad \mbox{implying} \quad \|f\|_{L^1} = 2a |S|,$$
where we take $a>0$ without loss of generality.
Moreover, the ``formal'' contribution to the total variation interpretation of the gradient is given by
$$ \|\nabla f \|_{L^1} = |\Gamma| (a - b) =|\Gamma|  \left(a + a\frac{|S|}{|\Omega \setminus S|} \right).$$
This implies
$$ \frac{  \|\nabla f \|_{L^1}  }{ \|f\|_{L^1}} =  \frac12 |\Gamma|  \left( \frac{1}{|S|} + \frac{1}{|\Omega\setminus S|}\right) =  \frac12 |\Gamma| \frac{|\Omega|}{|S| |\Omega \setminus S|}.$$
To be rigorous, convolution with a smooth mollifier shows that one can get arbitrarily close to the optimal constant with smooth functions. We will now show that, for a rectangle, the optimal spectral cut splits the domain via a straight line cut intersecting the longest sides at right angles.  The crucial ingredient here is that the argument does not appeal to symmetry,
is stable under perturbations and easily implies the same results for domains that are merely close to rectangles in the Gromov-Hausdorff distance.

\begin{lemma} \label{Lemma: rectangle result}
Let $R = [0,a] \times [0,b] \subset \mathbb{R}^2$  with $a > b$. The $1$-spectral cut is exactly at $\left\{a/2\right\} \times [0,b]$.
\end{lemma}
\begin{proof}
The cut in the middle of the longer side yields
$$ \frac{\mathcal{H}^1(\Gamma) |\Omega|}{|S| |\Omega \setminus S|} = \frac{b (ab)}{(ab)^2/4} = \frac{4}{a}.$$
It is clear that other spectral cuts touching the longer sides have $ \frac{\mathcal{H}^1(\Gamma) |\Omega|}{|S| |\Omega \setminus S|} \geq \frac{4}{a},$ since $\mathcal{H}^1(\Gamma)\geq b$ and $|S| |\Omega \setminus S|\leq (ab)^2/4$. 
\begin{center}
\begin{figure}[ht!]  % \label{fig:complete2}
\begin{tikzpicture}[scale = 2]
\draw [ultra thick] (0,0) -- (1.61,0) -- (1.61, 1) -- (0, 1) -- (0,0);
\draw [dashed] (0.8,0) -- (0.8, 1);
\draw [thick] (0.7,0) to[out=19,in=140] (1.61, 0.6);
%\draw [dashed] (1.61, 0.6)  to[out=0,in=90] (2.52, 0);
%\draw [dashed] (2.52, 0)  to[out=270,in=0] (1.61, -0.6);
%\draw [dashed] (1.61, -0.6)  to[out=180,in=270] (0.7, 0);
\node at (1.5,0.2) {$S$};
\node at (0.3,0.6) {$\Omega \setminus S$};
\end{tikzpicture}
\caption{The geometric construction in the proof of Lemma 1. }
\label{fig4}
\end{figure}
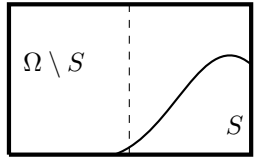
\end{center}

If $\Gamma$ touches two opposite sides, then
$$ \frac{\mathcal{H}^1(\Gamma) |\Omega|}{|S| |\Omega \setminus S|}  \geq b \frac{ |\Omega|}{|S| |\Omega \setminus S|} \geq \frac{b (ab)}{(ab)^2} \min_{0 \leq x \leq 1}{\frac{1}{x(1-x)}} \geq \frac{4}{a}.$$
Equality can only arise if the cut has length $b$ (forcing it to be a line) and $x=1/2$ (forcing a split into two domains of the same area). This characterizes
exactly the cut in the middle.

It remains to deal with the case where the spectral cut touches two adjacent sides; this case is illustrated in Figure \ref{fig4}. Let us assume that the enclosed domain is denoted by $S$, $|S| = x |\Omega| = x ab$ for some $0 < x < 1$, and $\Gamma = \partial S \cap \mbox{int}~R$ is the part of the boundary curve strictly inside the rectangle.
The isoperimetric inequality implies that
$$ \mathcal{H}^1(\Gamma) \geq  \sqrt{\pi}\sqrt{x a b},$$
and thus
$$ \frac{\mathcal{H}^1(\Gamma) |\Omega|}{|S| |\Omega \setminus S|} \geq \frac{ \sqrt{\pi x a b} ab }{a^2 b^2 x (1-x)} = \frac{\sqrt{\pi}}{\sqrt{x}(1-x)} \frac{1}{\sqrt{a}\sqrt{b}}.$$
An explicit computation shows that for all $0 < x < 1$
$$ \frac{1}{\sqrt{x}(1-x)}  \geq \frac{3\sqrt{3}}{2}$$
and therefore, using $b \geq a,$
$$ \frac{\mathcal{H}^1(\Gamma) |\Omega|}{|S| |\Omega \setminus S|} \geq \frac{3\sqrt{3 \pi}}{2} \frac{1}{a} \geq \frac{4.6}{a}.$$
This is always a constant fraction worse than the cut along the longest axes, concluding the argument.
\end{proof}

\begin{remark}
\label{rmk1}
It is easily seen that all aspects of the argument are stable under (even moderately large) perturbations of the domain in the sense of Gromov-Hausdorff distance.  To be specific, take a domain $Q$ satisfying Assumption \ref{rectangle_assumptions}, and denote the length of the shorter side of the associated rectangle $R$ by $h$. By Assumption \ref{rectangle_assumptions}, the Gromov-Hausdorff distance between $Q$
and $R$ is $d_{GH}(Q, R) \leq \varepsilon h$. There is clearly a straight line that is parallel to the shorter side of $R$ and bisects the area: the
bound on the Gromov-Hausdorff distance implies that this straight line has length at most $h + 2h\varepsilon$. Therefore,
$$   \inf_{\Gamma}{\frac{\mathcal{H}^1(\Gamma)}{|S| |\Omega \setminus S|}} \leq 4h + 8h \varepsilon.$$
Arguments as in Lemma 1 show that it has to cut the longer side roughly in the middle. 
\end{remark}

The second ingredient is a straightforward consequence of the isoperimetric inequality that we need to make quantitative.
\begin{center}
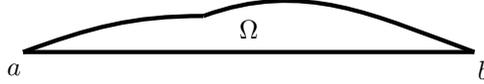
\begin{figure}[ht!] % \label{fig:complete3}
\begin{tikzpicture}[scale = 6]
\draw [ultra thick] (0,0) -- (1,0);
\node at (-0.02, -0.04) {$a$};
\node at (1.02, -0.04) {$b$};
\node at (0.5, 0.05) {$\Omega$};
\draw [ultra thick] (0,0) to[out=20,in=180] (0.4, 0.08);
\draw [ultra thick] (0.4,0.08) to[out=20,in=160] (1, 0);
\end{tikzpicture}
\caption{A curve slightly longer than the straight line.}
\label{fig:areagain}
\end{figure}
\end{center}

\begin{lemma} Consider a simply connected domain $\Omega$ as shown in Figure \ref{fig:areagain} that is comprised of a straight line segment between $a$ and $b$ and some arbitrary curved line segment enclosing a domain $\Omega$. For every $\varepsilon_0 >0$
there exists an $\varepsilon_1 > 0$ such that if
\begin{equation}
\label{assump1}
|\partial \Omega|  \leq  \left(1+\varepsilon_1\right) 2\|a - b\|,
\end{equation}
then
$$ |\Omega| \leq  \frac{1 + \varepsilon_0}{\sqrt{6}} \| a - b\|^{3/2} \sqrt{|\partial \Omega| - 2 \|a-b\|}.$$
\end{lemma}
\begin{proof} The isoperimetric inequality immediately implies that the optimal curve defining the Ratio Cut for $Q$ satisfying Assumption \ref{rectangle_assumptions} must take the form of a circular arc; otherwise, we could fix $a,b$ at the boundary of a circle
and use the novel shape to create a set in the plane that contains more area than the disk whose boundary has the same size. 

 It remains to compute the constants for the circular arc.
Choose a coordinate system with $a,b$ at $(0,0),(\|a-b\|,0)$ respectively. The opening angle $\alpha$ of the sector of the disk with radius $r>0$ that produces such an arc is given by
$$ \alpha = 2 \arcsin{\left( \frac{\|a-b\|}{2r} \right)},$$
where $0 < \alpha < \frac{\pi}{2}$ is sufficiently small by assumption \eqref{assump1}.
As a consequence, the length of the circular arc is given by
$$ r \alpha =  2r \arcsin{\left( \frac{\|a-b\|}{2r}\right)} \geq \|a-b\| + \frac{\|a-b\|^3}{24} \frac{1}{r^2},$$
and so
\begin{equation}
\label{pomegabd}
 |\partial \Omega| \geq 2\|a-b\| +   \frac{\|a-b\|^3}{24} \frac{1}{r^2}.
 \end{equation}
In particular, by assumption \eqref{assump1}, we have that 
$$  \frac{\|a-b\|^3}{24} \frac{1}{r^2} \leq  2\varepsilon_1\|a-b\| \qquad \mbox{and thus} \qquad \frac{\|a-b\|^2}{r^2}  \leq 48 \varepsilon_1.$$

The enclosed area captured by the line segment and the circular arc is 
$$\frac{ r^2 \pi \alpha}{2\pi} - r^2 \cos{\left(\frac{\alpha}{2}\right)} \sin{\left(\frac{\alpha}{2}\right)}.$$
%where we have assumed that $0 < \alpha < \frac{\pi}{2}$.
We have
$$ \frac{ r^2 \pi \alpha}{2\pi} =  r^2 \arcsin{\left( \frac{\|a-b\|}{2r} \right)} = r \frac{\|a-b\|}{2} + \frac{\|a-b\|^3}{48 r}  + \mbox{higher order terms}$$
as well as
$$ \cos{\left(\frac{\alpha}{2}\right)} = \sqrt{1 - \frac{\|a-b\|^2}{4r^2}} \quad \mbox{and} \quad \sin{\left(\frac{\alpha}{2}\right)} =  \frac{\|a-b\|}{2r}.$$
The higher order terms in the expansion are an infinite series in $\|a-b\|/r$ with exponents decaying fast enough. For $\varepsilon_1$ sufficiently
small depending on $\varepsilon_0$, we can bound
$$ \mbox{higher order terms}  \leq \varepsilon_0  \frac{\|a-b\|^3}{48 r}  .$$
Altogether this implies that
$$\mbox{area} \leq (1 + \varepsilon_0) \frac{\|a-b\|^3}{12r}$$
and plugging in the bound on $1/r$ from \eqref{pomegabd} gives
$$\mbox{area} \leq  \frac{1 + \varepsilon_0}{\sqrt{6}} \| a - b\|^{3/2} \sqrt{|\partial \Omega| - 2 \|a-b\|}.$$
\end{proof}

\subsection{Proof of Theorem 1}

\begin{proof} 
We assume without loss of generality $|Q| = 1$. For this proof we minimize, over all possible set partitions $Q = S \cup (Q \setminus S)$, the quotient
$$ \inf_{\Gamma} \frac{\mathcal{H}^1(\Gamma) }{|S| | Q \setminus S|},$$
where $\Gamma=\overline S\cap \overline{Q \setminus S}$.

Lemma \ref{Lemma: rectangle result} and Remark \ref{rmk1} show that the cut must intersect opposite sides; if the aspect ratio is large enough, then these opposite sides will be the longest sides. 

%Lemma \ref{Lemma: rectangle result} shows that the cut has to intersect the two longer sides. If $R$ is close to a square in the Gromov-Hausdorff distance, then the argument in Remark \ref{rmk1} show, the cut must intersect two opposite sides. 

%Let us denote the length of the shorter side of $R$ by $h$. By Assumption \ref{rectangle_assumptions}, the Gromov-Hausdorff distance between $Q$
%and $R$ is $d_{GH}(Q, R) \leq \varepsilon h$. There is clearly a straight line that is both parallel to the shorter side and bisects the volume: the
%bound on the Gromov-Hausdorff distance implies that this straight line has length at most $h + 2h\varepsilon$. Therefore
%$$   \inf_{E}{\frac{\mathcal{H}^1(E)}{|S| |\Omega \setminus S|}} \leq 4h + 8h \varepsilon.$$
%Arguments as in Lemma 1 show that it has to cut the
%longer side roughly in the middle. 

We now show that the distance to the axis of symmetry of $R$ is of order $\lesssim \sqrt{\varepsilon}$.
Since the Gromov-Hausdorff distance between the domain and the rectangle is $\varepsilon$, by the same calculation in Remark \ref{rmk1} we see that any cut has to have length at least
$ h (1- 2\varepsilon)$, where $h$ is the length of the shorter side of the associated rectangle $R$. Note that the shortest line between two points may not make the optimal Ratio Cut, though, as that may result in a less favorable area splitting. Hence, one possibility is that we get a gain from having a slightly longer curve to encompass a more favorable area, but then by the isoperimetric inequality, it can be seen that circular arcs are favorable to any other configuration in terms of length to area trade-offs.  A key example of this is the trapezoidal domain with angled top and bottom boundary curves. Therefore the optimal spectral cut $Q = A \cup B$ satisfies
$$  \frac{h(1-2\varepsilon)}{|A| |B|} = \frac{h(1-2\varepsilon)}{|A| (1-|A|)} \leq  \frac{\mathcal{H}^1(\Gamma)}{|S| |Q \setminus S|} \leq 4h(1 + 2 \varepsilon)$$
for any $\Gamma$.
Thus, for $\varepsilon \leq 0.1$
\begin{equation}
\frac{1}{|A| (1-|A|)} \leq 4\frac{1 + 2\varepsilon}{1 - 2 \varepsilon} \leq 4  + 20\varepsilon.\label{bound of A(1-A)}
\end{equation}
When combined with the elementary inequality
$$ \frac{1}{x(1-x)}  \geq   4 + 16\left(x-\frac{1}{2}\right)^2,$$
we have that the optimal spectral cut yields two sets satisfying
$$ \frac{1}{2} - \sqrt{\frac{5\varepsilon}{4}} \leq |A| \quad \mbox{and}\quad |B| \leq \frac{1}{2} + \sqrt{\frac{5\varepsilon}{4}}.$$
The Lipschitz bound then ensures that the spectral cut's intersection with $\partial Q$ occurs in a $\sqrt{\epsilon}$ small neighborhood of $Q$'s axis of symmetry. 

Since the functional itself only contains the length $\mathcal{H}^1(\Gamma),$ as well as a term depending on the partition of the areas $|S||\Omega\setminus S|$, we can conclude that the optimal $\Gamma$ is a circular arc; otherwise, we can minimize the length of the curve while keeping the bounded area fixed.  Indeed, this boils down to the question of minimizing the arc-length of a curve of fixed area,
\[
\text{arg min}_{y = \gamma (x)} \left\{  \int_a^b \sqrt{1 + (\gamma'(x) )^2} dx \ \bigg| \ \int_a^b |\gamma| dx \,\,\mbox{ fixed.}\right\},
\] 
which is minimized by a circular arc via the isoperimetric inequality.
We already know from above that $\mathcal{H}^1(\Gamma) \leq (1+2\varepsilon) h$, so Lemma 2 implies that the area captured by the curved
arc satisfies
$$ \mbox{captured area} \leq  \frac{1 + c_{\varepsilon}}{\sqrt{6}} h^{3/2} \sqrt{\mathcal{H}^1(\Gamma) - h} \,\leq  \frac{1 + c_{\varepsilon}}{\sqrt{3}} h^2 \sqrt{\varepsilon} ,$$
where $c_{\varepsilon} > 0$ depends on $\varepsilon$, but does tend to 0 as $\varepsilon$ tends to $0$. 

In summary, the above conclusions show that the Ratio Cut $\Gamma$ must be a circular arc in a $\sqrt{\epsilon}$ neighborhood of the axis of symmetry of the reference rectangle from Assumption $1$. 
\end{proof}

\section{Proof of Theorem \ref{thm2} for Curvilinear Quadrilaterals with Parabolic Top and Bottom Curves}
\label{sec:pf3}

To prove Theorem $2$, we will first analyze how the Ratio Cut depends on parameters determining the domain $Q$ and circular arc cut $\Gamma$ for a true quadrilateral. Then, we will demonstrate for a simple example of a curvilinear trapezoid with one parabolic edge and three flat edges that the curvature has a smaller order impact on the Ratio Cut than the trapezoidal feature. Lastly, we will prove the theorem for even more general domains with parabolic top and bottom bounding curves. These domains will be shown to be generic in section 5, at least to leading order in how the Ratio Cut depends upon the structure of the curves that make up the longest sides of our approximately parabolic curvilinear rectangles. 

Consider the quadrilateral $Q$ determined by the vertices
$$ (0,0), (x_1, a), (1,0) \quad \mbox{and} \quad (1+x_2, a+x_3),$$
where we assume that $ 0 < a < 1$ (implying that we have normalized the quadrilateral and put the longer side on the $x-$axis). The quantities $x_1, x_2, x_3$
are perturbation parameters and assumed to be small, in the sense that $|x_i| \ll \min\left\{a,1- a\right\}$. The condition $|x_i| \ll a$ is clear; we want the perturbation to be small with respect to length and height of the rectangle. The other condition is slightly more subtle: if it is not satisfied, then the quadrilateral might be close to a square and the location of the spectral cut will depend nonlinearly on the perturbation parameters. The cut is still going to be a circular arc nearly bisecting the domain, but a slight perturbation of $x_1, x_2, x_3$ may send $\Gamma$  from being nearly vertical to nearly horizontal (see Fig. \ref{fig:quad}).
The same arguments used to prove Theorem 1 show that the spectral cut will be a circular arc connecting two points ${\bf q} = (q,0)$ and ${\bf p} = (p,y(p))$.

\begin{center}
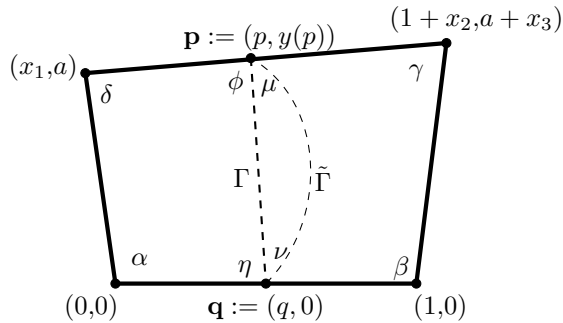
\begin{figure}[ht!]
\begin{tikzpicture}[scale=4]
\draw [ultra thick] (0,0) -- (1,0) -- (1.1, 0.8) -- (-0.1, 0.7) -- (0,0);
\filldraw (0,0) circle (0.015cm);
\node at (-0.08, -0.08) {(0,0)};
\node at (0.08, 0.08) {$\alpha$};
\filldraw (1,0) circle (0.015cm);
\node at (0.95, 0.05) {$\beta$};
\node at (1.08, -0.08) {(1,0)};
\filldraw (1.1,0.8) circle (0.015cm);
\node at (1.2, 0.8+0.08) {($1+x_2$,$a+x_3$)};
\filldraw (-0.1,0.7) circle (0.015cm);
\node at (-0.24, 0.8-0.08) {($x_1$,$a$)};
\filldraw (0.5,0) circle (0.015cm);
\node at (0.5, -0.08) {${\bf q}:=(q,0)$}; 
\filldraw (0.45,0.75) circle (0.015cm);
\node at (0.48,0.82) {${\bf p}:=(p,y (p)) $}; 
\draw [thick,dashed] (0.45, 0.75) -- (0.5,0);
\draw [thin,dashed] (0.5,0) to[out=45,in=330] (0.45, 0.75);
%\draw[red] (0.5,0) arc (0:180:3);
%\node at (0.21, 0.35) {$S$};
\node at (0.42, 0.35) {$\Gamma$};
\node at (0.69, 0.35) {$\tilde{\Gamma}$};
\node at (1, 0.7) {$\gamma$};
\node at (-0.03, 0.63) {$\delta$};
\node at (0.43, 0.05) {$\eta$};
\node at (0.55, 0.1) {$\nu$};
\node at (0.4, 0.67) {$\phi$};
\node at (0.51, 0.65) {$\mu$};
\end{tikzpicture}
\caption{A general quadrilateral $Q$ after rescaling.}
\label{fig:quad}
\end{figure}
\end{center}

\subsection{Circular Arc to Triangle in terms of the sector angle}

Let us assume the cut, $\tilde E$, is a circular arc from $\bf p$ to $\bf q$, where $\bf p$ and $\bf q$ appear nearly opposite each other on the longest sides.   We will take $|\tilde \Gamma| = r \theta$ for some radius $r>0$ and sector angle $\theta\geq0$, to be determined, and observe then the area contained between the line $\Gamma$ connecting $\bf p$ and $\bf q$ and the circular arc $\tilde \Gamma$, which we will call $\tilde \Omega$, satisfies
\[
| \tilde \Omega | = \frac{r^2 \theta}{2} - \frac{\| {\bf p}- {\bf q}\|^2}{4 \tan (\theta/2)}.
\]
We also have that
\[
r = \frac{\| {\bf p}- {\bf q} \|}{2 \sin (\theta/2)},
\]
from which we conclude, after Taylor expanding in $\theta$, that
\[
| \tilde \Omega | = \|{\bf p}- {\bf q}\|^2 \left( \frac{\theta^2}{48} + O( \theta^4) \right).
\]
Also, Taylor expanding once more, we have
\[
|\tilde \Gamma| = r \theta = \| {\bf p}- {\bf q}\| \left( 1 + \frac{\theta^2}{24} + O( \theta^4) \right).
\]

With these identities in hand, we can proceed to explore how the ratio cut depends upon the curves defining the longest aspect ratio sides of our curvilinear quadrilateral.  To illustrate how to compute the dependence of the cut locally on the parametrization of a curve, we will first work with a toy model that is flat on $3$ sides and has a smooth quadratic curve on one of the long sides.  
%\wh{Check this for continuity after (if) we add an outline for this section.}

\subsection{The parabolic trapezoid}
\label{partrap}

We take a domain that is a parabolic trapezoid bounded by a straight line from $(0,0)$ to $(1,0)$, a straight line from $(0,0)$ to $(0,1/2)$, the straight line from $(1,0)$ to $(1, 1/2+a)$ and the curve $y(x) = \epsilon x^2 + (a - \epsilon) x + 1/2$; note that, for convenience, we have chosen our aspect ratio to be approximately $2$.  Our goal is to track how the cut changes as a function of $\epsilon,a$ close to $0$.   The Ratio Cut will intersect the bottom and top boundary curves at the points $(q,0)$ and $(p,y(p))$ respectively, with the cut itself a circular arc.  Splitting the domain by a circular arc, we can decompose one of the cut domains into two curvilinear quadrilaterals, see Figure \ref{fig:quad2} for an illustration.

\begin{center}
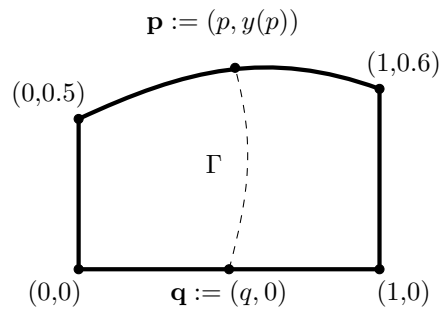
\begin{figure}[ht!]
\begin{tikzpicture}[scale=4]
\draw [ultra thick] (0,1/2)--(0,0) -- (1,0) -- (1, 1/2+0.1) ;%-- (0, 1/2) -- (0,0);
\draw [ultra thick] (0,0.5) to[out=25,in=160] (1, 0.6);
\filldraw (0,0) circle (0.015cm);
\node at (-0.08, -0.08) {(0,0)};
%\node at (0.08, 0.08) {$\alpha$};
\filldraw (1,0) circle (0.015cm);
%\node at (0.95, 0.05) {$\beta$};
\node at (1.08, -0.08) {(1,0)};
\filldraw (1,1/2+0.1) circle (0.015cm);
\node at (1.08, 1/2+0.1+0.08) {($1$,$0.6$)};
\filldraw (0,0.5) circle (0.015cm);
\node at (-0.1, 0.5+0.08) {($0$,$0.5$)};
\filldraw (0.5,0) circle (0.015cm);
\node at (0.5, -0.08) {${\bf q}:=(q,0)$}; 
\filldraw (0.52,0.67) circle (0.015cm);
\node at (0.48,0.82) {${\bf p}:=(p,y (p)) $}; 
%\draw [thick,dashed] (0.52, 0.67) -- (0.5,0);
\draw [thin,dashed] (0.5,0) to[out=75,in=285] (0.52, 0.67);
%\draw[red] (0.5,0) arc (0:180:3);
%\node at (0.21, 0.35) {$S$};
\node at (0.45, 0.35) {$\Gamma$};
\end{tikzpicture}
\caption{A parabolic trapezoid $Q$ as a simplified model.}
\label{fig:quad2}
\end{figure}
\end{center}

It can be easily seen that to compute the Ratio Cut, we can split the domains into components given by the region $[0,p]$, a triangle with base $[p,q]$, and the cap of a circular arc that is either added or subtracted depending upon orientation ($p<q$ or $q<p$ and $\theta > 0$ or $\theta < 0$ respectively). Decomposing the domain into these components, we compute
\begin{align*}
&A_{total}  = \int_0^1 y(x) dx = \frac{\epsilon}{3} + \frac{a-\epsilon}{2} + \frac12, \ \ \text{(Total Area)} \\
&A_1 (p,q,\theta)  = \int_0^p y(x) dx + \frac12 y(p) (q-p)+\frac12 R^2 (p,q,\theta) \theta \\
& \hspace{4cm} - \frac14 \left[   (q-p)^2 + y(p)^2 \right] \cot \left( \frac{\theta}{2} \right), \ \ \text{(Left Area)} \\
& A_2 (p,q,\theta)  = A_{total} - A_1 (p,q,\theta), \ \ \text{(Right Area)} 
\end{align*}
where $R(p,q,\theta)$ is the radius of the circle, given by
\begin{align}
R(p,q,\theta)  = \frac{\sqrt{  (q-p)^2 + y(p)^2 }}{ 2 \sin \left( \frac{\theta}{2} \right)}\,.
\end{align}
The above quantities give the ratio cut, denoted by $\text{RC} (p,q,\theta)$:
\begin{align*}
 \text{RC} (p,q,\theta) = \frac{ R(p,q,\theta) \theta}{ A_1 (p,q,\theta) A_2 (p,q,\theta)}.
\end{align*}
It follows via direct calculation that $ \text{RC} (p,q,\theta)$ is a smooth function of $(p,q,\theta)$ in a neighborhood of $(1/2,1/2,0)$. Indeed, since $\frac12 R^2 (p,q,\theta) \theta - \frac14 \left[   (q-p)^2 + y(p)^2 \right] \cot \left( \frac{\theta}{2} \right)\to 0$ and $R(p,q,\theta) \theta\to \sqrt{  (q-p)^2 + y(p)^2 }$ when $\theta\to 0$, $\text{RC} (p,q,\theta)$ is smooth in a neighborhood of $(1/2,1/2,0)$. Hence we can explore behaviors nearby using the Implicit Function Theorem.    %\jlm{Include the simple argument that $\text{RC}$ is smooth in the parameters to justify Taylor Series truncation??}

As an illustrative calculation, let us simply take the full quadratic approximation in $\theta$, $p-\frac12$, $q-\frac12$, $a$ and $\epsilon$ to the Ratio Cut. A direct calculation gives:
\begin{small}
\begin{align*}
& \text{RC} (p,q,\theta)  {\color{red}=} \left(  8 + 24 \left(q - \frac12 \right)^2 - 16 (q-\frac12) (p-\frac12) + 24 (p - \frac12)^2 \right) \\
& - a \left(8 + 8\left(q - \frac12 \right) - 8 \left(p - \frac12 \right) - 56\left(q - \frac12 \right) ^2 - 56 \left(p - \frac12 \right)^2 + 80 \left(q - \frac12 \right)  \left(p - \frac12 \right)    \right) \\
& + \epsilon \left(  \frac43 + \frac{52}{3} \left(q - \frac12 \right)^2 + \frac{100}{3} \left(p - \frac12 \right)^2    - 40 \left(q - \frac12 \right)  \left(p - \frac12 \right)  \right) + \frac43 \theta \left(   \left(q - \frac12 \right) + \left(p - \frac12 \right) \right) \\
& + a^2  \left( 10 + 16 \left(q - \frac12 \right) - 16 \left(p - \frac12 \right)  + 120 \left(q - \frac12 \right)^2 + 104 \left(p - \frac12 \right)^2  -192 \left(q - \frac12 \right) \left(p - \frac12 \right) \right)  \\
& + \epsilon^2 \left(  \frac{122}{9}  \left(q - \frac12 \right)^2 + \frac{218}{9} \left(p - \frac12 \right)^2  - \frac{284}{9} \left(q - \frac12 \right)  \left(p - \frac12 \right)   \right)  +  \frac{7}{18} \theta^2   \\
& -a \epsilon \left(  \frac83 + \frac{8}{3} \left(q - \frac12 \right) - 8 \left(p - \frac12 \right)   + 72 \left(q - \frac12 \right)^2 + 104 \left(p - \frac12 \right)^2 - \frac{464}{3} \left(q - \frac12 \right)  \left(p - \frac12 \right)   \right)  \\
&  -  a \theta \left(  \frac23 + 8 \left(q - \frac12 \right)^2 - \frac{32}{3}  \left(q - \frac12 \right)  \left(p - \frac12 \right)  \right)  - \frac89 \epsilon \theta \left( \left(q - \frac12 \right)  + \left(p - \frac12 \right)   \right)  
\end{align*}\end{small} 
up to a higher order error, when $a$ and $\epsilon$ are sufficiently small.
While this may not look so useful, we get a great deal of information by looking at the system when $a = \epsilon = 0$.  In such a case, the equations for a critical point in $p,q,\theta$ become
\begin{align*}
48 \left(q - \frac12 \right)  - 16 \left(p - \frac12 \right) + \frac43 \theta & = 0 \\
-16 \left(q - \frac12 \right)  + 48 \left(p - \frac12 \right) + \frac43 \theta & = 0 \\
\frac43 \left(q - \frac12 \right)  + \frac43 \left(p  - \frac12 \right) + \frac79 \theta & = 0.
\end{align*}
Thus, the Jacobian matrix is
\begin{equation*}
J_{a=0,\epsilon=0} = \left[  \begin{array}{rrr}
48 & -16 & \frac43 \\
-16 & 48 & \frac43 \\
\frac43 & \frac43 & \frac79
\end{array}
\right],
\end{equation*}
which is non-singular.  As a result, we observe that if $a = \epsilon = 0$, the optimal solutions is $p = q = \frac12$, $\theta = 0$.  We know this from symmetry arguments, but now we also have set ourselves up for an application of the Implicit Function Theorem in order approximate the RC for near rectangular domains.  

We next observe what happens if $a \neq 0$, $\epsilon = 0$. This gives the modified system
\begin{equation*}
J_{a,\epsilon=0} \left[ \begin{array}{c}
\left(q - \frac12 \right)  \\
\left(p - \frac12 \right) \\
\theta 
\end{array} \right] = a \left[ \begin{array}{r}
8 \\
-8 \\
\frac23
\end{array} \right]  + \text{Quadratic Error in $a, q-\frac12, p-\frac12$},
\end{equation*}
where
\begin{equation*}
J_{a,\epsilon=0} = \left[  \begin{array}{ccc}
48 -112 a & -16 + 80 a& \frac43 \\
-16 +80 a & 48 -112 a & \frac43 \\
\frac43 & \frac43 & \frac79
\end{array}
\right],
\end{equation*}
Hence, $\vec v$ satisfying $J_{a,\epsilon=0} \vec{ v }= \vec{0}$ is given by 
\begin{equation*}
\vec{v} = a J_{a=0,\epsilon=0}^{-1} \left[ \begin{array}{c}
8 \\ -8 \\ \frac23
\end{array} \right] + O(a^2) = a \left[  \begin{array}{c}
\frac{1}{12} \\ -\frac16 \\ 1 \end{array}  \right] + O(a^2).
\end{equation*}
Note that the distance between $p$ and $q$ is then $\frac{a}{4} < a$, with the maximal amplitude of the bulge from the circular arc of size 
\[  \frac{a}{8}\sqrt{y(p)^2+(p-q)^2} < \frac{a}{8}. \]  
Since the terms that are linear in $\epsilon$ are quadratic or higher in $q-1/2,p-1/2$, a similar analysis including $\epsilon$ shows that the curvature of the parabolic curve is actually a lower order deformation for the Ratio Cut than the trapezoidal deflection.  Importantly, this argument demonstrates that the trapezoidal deflection is decreasing in the new cut domain.  Though the new cut domain will be closer to aspect ratio $1$, the overall deflections are still decreasing on subsequent domains.

Using the Implicit Function Theorem, we are able to compute comparable results for the full Ratio Cut.  Indeed, to turn this into a rigorous argument, we need first observe that the minimum Ratio Cut at $a=\epsilon = 0$ is uniquely $p=q=\frac12$, $\theta=0$, which is easily seen by looking at the Hessian.  Then, we look at $ \text{RC} (p,q,\theta; a, \epsilon)$ as a map from $\mathbb{R}^5 \to \mathbb{R}^3$, and use the Implicit Function Theorem to construct the desired local map from $(p,q,\theta;\,a,\epsilon)$ in an open set around $(a,\epsilon) = (0,0)$. 

\subsection{The Ratio Cut with sides given by parabolic approximations}
\label{arcs}

Now, we proceed to handle a more general family of domains.  Given our assumption on the smoothness of the curves, we can assume that the top and bottom curves are approximated by quadratic curves to high accuracy near points of intersection with the ratio cut.  Specifically, let us consider an arbitrary domain $Q$ that can be approximated (in the Gromov-Hausdorff sense) by a parabolic trapezoid $Q_0$ with vertices $(0,0),(0,\frac{1}{2}+a_1),(1,\frac{1}{2}+a_2),$ and $(1,0)$. Note, the inclusion of $a_1$ and $a_2$ here will allow us to vary the aspect ratio. We fix the width of $Q_0$ to $1$, however, as can always be done by a scaling of the domain.  The top paraboloid of $Q_0$ is parametrized as 
$$
y_{T}(x) = \epsilon_{t}x^2 + (a_2 - a_1 - \epsilon_{t})x + a_1 + \frac{1}{2},
$$ 
and the bottom paraboloid is parametrized as 
$$
y_{B}(x) = \epsilon_{b}x^2 -\epsilon_b x.
$$ 
We assume that $Q$ and $Q_0$ differ by two sufficiently small and bounded ``black-box'' regions on the left and right. % not enclosed by the region considered above. 
The areas of these two small regions will be denoted $A_{\texttt{WL}}$ and $A_{\texttt{WR}}$.

\begin{center}
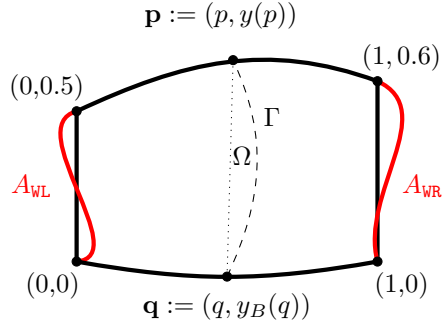
\begin{figure}[ht!]
\begin{tikzpicture}[scale=4]
\draw [ultra thick] (0,1/2)--(0,0) ;

\draw [red,ultra thick] (0,0) to[out=5,in=195] (0, 0.5);
\node[red] at (-0.15, 0.25) {$A_{\texttt{WL}}$};

\draw [ultra thick] (1,0) -- (1, 1/2+0.1) ;

\draw [red,ultra thick] (1,0) to[out=105,in=335] (1, 0.6);
\node[red] at (1.15, 0.25) {$A_{\texttt{WR}}$};

\draw [ultra thick] (0,0.5) to[out=25,in=160] (1, 0.6);
\draw [ultra thick] (0,0) to[out=-10,in=190] (1, 0);
\filldraw (0,0) circle (0.015cm);
\node at (-0.08, -0.08) {(0,0)};
%\node at (0.08, 0.08) {$\alpha$};
\filldraw (1,0) circle (0.015cm);
%\node at (0.95, 0.05) {$\beta$};
\node at (1.08, -0.08) {(1,0)};
\filldraw (1,1/2+0.1) circle (0.015cm);
\node at (1.08, 1/2+0.1+0.08) {$(1,0.6)$};
\filldraw (0,0.5) circle (0.015cm);
\node at (-0.1, 0.5+0.08) {($0$,$0.5$)};
\filldraw (0.5,-0.05) circle (0.015cm);
\node at (0.5, -0.15) {${\bf q}:=(q,y_B(q))$}; 
\filldraw (0.52,0.67) circle (0.015cm);
\node at (0.48,0.82) {${\bf p}:=(p,y (p)) $}; 
%\draw [thick,dashed] (0.52, 0.67) -- (0.5,0);
\draw [thin,dashed] (0.5,-0.05) to[out=65,in=295] (0.52, 0.67);
\draw [dotted] (0.5,-0.05) -- (0.52, 0.67) ;
\node at (0.65, 0.48) {$\Gamma$};
\node at (0.55, 0.35) {$\Omega$};
\end{tikzpicture}
\caption{A more general domain $Q$ approximated by a parabolic trapezoid.}
\label{fig:quad3}
\end{figure}
\end{center}

We start by preparing quantities associated to $Q_0$. A circular arc $\Gamma$ passing through the points $(p,y_{T}(p))$ and $(q, y_B(q))$, with angle $\theta$, cuts the parabolic trapezoid into a left and right domain, denoted by $S$ and $Q_0\backslash S$. %; the areas of these regions will be $A_{L}$ and $A_R$, with the total area of the trapezoid denoted $A_T$. 
	To compute an equivalent analytic expression for the ratio cut, we use Stoke's theorem %$\int_\Omega d\omega = \int_{\partial \Omega} \omega$ 
	to compute the left area $A_L = |S|$ and total $A_T = |Q_0|$, which we use to compute the right area $A_R = |Q_0\backslash S| = |Q_0|-|S|$. In particular, 
	$$
	A_T = |Q_0| = \int_{Q_0} dA = \frac{1}{2} \int_{\partial Q_0} xdy - ydx,
	$$
	where we integrate along the left vertical boundary $\{(0,t): 0\leq t\leq \frac{1}{2}+a_1\}$, along the top parabolic curve $\{(x,y_T(x)): 0\leq x\leq 1 \},$ along the right vertical boundary $\{(1,t): \frac{1}{2}+a_2\leq t\leq 0 \}$ (with the indicated orientation), and finally along the bottom parabolic curve $\{(x,y_B(x)): 1\leq x\leq 0 \}$ (with the indicated orientation). 
	We can compute the (indefinite) integrals for the two parabolic pieces fairly easily:
	\begin{align*}
		dy_T 	&= (2\epsilon_t x + a_2-a_1-\epsilon_t)dx,\\
		dy_B 	&= (2\epsilon_b x -\epsilon_b)dx,
	\end{align*}
	and so
	\begin{align*}
		\frac{1}{2} \int xdy_T - y_T dx 	&= \frac{1}{2} \int x(2\epsilon_t x + a_2-a_1-\epsilon_t)dx - (\epsilon_{t}x^2 + (a_2 - a_1 - \epsilon_{t})x + a_1 + \frac{1}{2})dx\\
			&= \frac{1}{2}\int \epsilon_t x^2-(a_1+\frac{1}{2}) dx
			= \frac{\epsilon_t}{6}x^3 - \left(\frac{a_1}{2}+ \frac{1}{4} \right)x;\\
		\frac{1}{2} \int xdy_B - y_B dx 	&= \frac{1}{2}\int x(2\epsilon_b x -\epsilon_b)dx - (\epsilon_{b}x^2 + -\epsilon_b x)dx\\
			&= \frac{1}{2} \int \epsilon_b x^2 dx
			= \frac{\epsilon_b}{6}x^3.
	\end{align*}
The $1$-forms for the vertical components are 
	\begin{align*}
		\frac{1}{2} \int 0dt - t d(0) &= \frac{1}{2} \int 0 = 0,\\
		\frac{1}{2} \int 1 dt - t d(1) 	&= \frac{1}{2} \int dt - 0 = \frac{t}{2}.
	\end{align*}
	Putting it all together, we have
	\begin{align*}
		A_T &= 0 + \left[ \frac{\epsilon_t}{6}x^3 - \left(\frac{a_1}{2}+ \frac{1}{4} \right)x \right]_{x=0}^{x=\frac{1}{2}+a_1} + \left[\frac{t}{2} \right]_{t=\frac{1}{2}+a_2}^{0} + \left[ \frac{\epsilon_b}{6}x^3 \right]_{x=1}^{x=0}\\
			&= \frac{\epsilon_t a_1^3}{6} + \frac{\epsilon_t a_1^2}{4} + \frac{\epsilon_t a_1}{8} + \frac{\epsilon_t}{48} -\frac{\epsilon_b}{6} - \frac{a_2}{2} - \frac{a_1^2}{2} - \frac{a_1}{2} - \frac{3}{8}.
	\end{align*}
	
	For $A_L = |S|$, we use Stokes' theorem again but split $S$ into two parts divided by the straight line connecting $(p,y_t(p))$ and $(q,y_b(q))$; these parts are a curvilinear quadrilateral and circular cap. The area of the curvilinear quadrilateral is computed by Stokes' theorem, while the area of the circular cap is determined by an elementary formula. Adding these two quantities gives us $A_L$.

	The line connecting $(p,y_T(p))$ to $(q,y_B(q))$ is parametrized as $\{ (1-t)(p,y_T(p)) + t(q,y_B(q)): 0\leq t \leq 1 \},$ which componentwise becomes
	$$ 
	( (1-t)p + t q, (1-t)y_T(p) + ty_B(q)), 
	$$
	and so the integral along this straight line becomes
	\begin{align*}
		\frac{1}{2} \int xdy - ydx 	=\,& \frac{1}{2} \int ((1-t)p + t q)d((1-t)y_T(p) + ty_B(q))\\ &\quad - ((1-t)y_T(p) + ty_B(q))d((1-t)p + t q)\\
			=\,& \frac{1}{2} \int ((1-t)p + t q)(-y_T(p) dt + y_B(q) dt)\\
				&\quad - ((1-t)y_T(p) + ty_B(q)) ( -p dt + q dt)\\
			=\,& t q (p (\epsilon_b q-\epsilon_b-\epsilon_t p+\epsilon_t)-a_2 (p+1))\,,
	\end{align*}
	where we denote the last quantity by $SL(t)$. %, so that $$SL(t) = t q (p (\epsilon_b q-\epsilon_b-\epsilon_t p+\epsilon_t)-a_2 (p+1)).$$
	
	Given two points $p$ and $q$, and an angle $\theta$, the radius of the circle containing points $p$ and $q$ separated by angle $\theta$ is 
	$$
	R = \frac{\|p-q\|_2}{2\sin(\frac{\theta}{2})}.
	$$ 
	Thus, the area of the circular segment is	{\allowdisplaybreaks
	$$
	|\Omega|= \frac{R^2}{2}(\theta - \sin(\theta)) = \frac{(p - q)^2 + (y_T(p) - y_B(q))^2}{2}(\theta - \sin(\theta)),
	$$ 
	and so
	\begin{align*}
		A_L =\,& 0 + \left[ \frac{\epsilon_t}{6}x^3 - \left(\frac{a_1}{2}+ \frac{1}{4} \right)x \right]_{x=0}^{x=p} + [SL(t)]_{t=0}^{t=1} + \left[ \frac{\epsilon_b}{6}x^3 \right]_{x=q}^{x=0} + |\Omega|\\
		=\,& \frac{1}{6} \left( -\epsilon_b q^3 - \frac{3}{2}(1+2a_1)p + \epsilon_t p^3 + 6q (-a_2(1+p) + p( - \epsilon_b + \epsilon_t + \epsilon_t q - \epsilon_t p))\right.\\
		& \left.+ 3\left( (q-p)^2 + \left( a_2  + \epsilon_b q - \epsilon_b q^2 + a_2p - \epsilon_t p + \epsilon_t p^2\right)^2\right) (\theta - \sin(\theta))\right).
	\end{align*}
	We can use this to get the area of the right portion:
	\begin{align*}
		A_R =& A_T - A_L \\
			=& \frac{1}{48} \bigg(-18-24 a_1-24 a_1^2-24 a_2-8 \epsB+\epsT+6 a_1 \epsT+12 a_1^2 \epsT\\
			&+8 a_1^3 \epsT+8 \epsB \tBott^3+12 (1+2 a_1) \tTop-8 \epsT \tTop^3\\
			&-48 \tBott \big(-a_2 (1+\tTop)+\tTop (\epsT+\epsB (-1+\tBott)-\epsT \tTop)\big)\\
			&-24 \big((\tBott-\tTop)^2+(a_2+\epsB \tBott-\epsB \tBott^2+a_2 \tTop-\epsT \tTop+\epsT \tTop^2)^2\big) (\theta-\sin(\theta))\bigg).
	\end{align*}
	}
	Finally, the length of the ratio cut takes the form $$|\Gamma| = R\theta = \frac{\|(p-q, y_T(p)-y_B(q))\|_2 \frac{\theta}{2}}{\sin(\frac{\theta}{2})} = \frac{\|(p-q, y_T(p)-y_B(q))\|_2}{\sinc(\frac{\theta}{2})}.$$
	With the above pieces from $Q_0$, the ratio cut of $Q$ %corresponding to this choice of $p, q,$ and $\theta$, is defined as 
%	$$
%	RC(p, q, \theta) = \frac{|\Gamma|}{|S||Q_0\backslash S|},
%	$$ 
%	where $|\Gamma|$ is the arc length of $\Gamma$.
%
%	%
	%So far, we have 
	can be expressed in terms of variables $p, q,$ and $\theta$, together with parameters $a_1,a_2, \epsT, \epsB, A_{\texttt{WL}}, A_{\texttt{WR}}$ that determine the domain.  Specifically, by letting $\sigma = (a_1,a_2, \epsT, \epsB, A_{\texttt{WL}}, A_{\texttt{WR}})$ be the parameter vector, we can express the ratio cut as 
	\begin{align*}
		RC(q, p, \theta; \sigma) 	&= \frac{\|(p-q, y_T(p)-y_B(q))\|_2}{\sinc(\frac{\theta}{2})(A_L + A_{\texttt{WL}})(A_R+A_{\texttt{WR}})}.
	\end{align*}
	The series expansion of $RC(q, p, \theta;\sigma)$ near $\left( \frac{1}{2}, \frac{1}{2}, 0;   \vec{0} \right)$ is
	\begin{align*}
		RC 	=& 8 + 24 \tBCent^2 - 16 \tBCent \tTCent + \frac{4}{3} \tBCent \Th\\
			&+ 24 \tTCent^2 + \frac{4}{3} \tTCent \Th + \frac{7}{18}\Th^2\\
			&+ a_1 p_{a_1}(q, p, \theta) + a_2 p_{a_2}(q, p, \theta) + a_3 p_{a_3}(q, p, \theta)\\
			&+ \epsilon_t p_{\epsilon_t}(q, p, \theta) + \epsilon_b p_{\epsilon_b}(q, p, \theta)\\
			&+ A_{\texttt{WL}} p_{A_{\texttt{WL}}}(q, p, \theta) + A_{\texttt{WR}} p_{A_{\texttt{WR}}}(q, p, \theta)\\
			&+ a_1^{\,2} p_{a_1 a_1}(q, p, \theta) + a_1a_2 p_{a_1 a_2}(q, p, \theta) + a_1a_3 p_{a_1 a_3}(q, p, \theta)\\
			&+ a_2^{\,2} p_{a_2a_2}(q, p, \theta) + a_2a_3 p_{a_2a_3}(q, p, \theta) + a_3^{\, 2} p_{a_3a_3}(p, q, \theta)\\
			&+ a_1A_{\texttt{WL}} p_{a_1A_{\texttt{WL}}}(q, p, \theta) + a_1A_{\texttt{WR}} p_{a_1A_{\texttt{WR}}}(q, p, \theta)\\
			&+ a_1 \epsT p_{a_1 \epsT}(q, p, \theta) + a_1 \epsB p_{a_1 \epsB}(q, p, \theta)\\
			&+ a_2 A_{\texttt{WL}} p_{a_2A_{\texttt{WL}}}(q, p, \theta) + a_2A_{\texttt{WR}} p_{a_2 A_{\texttt{WR}}}(q, p, \theta)\\
			&+ a_2 \epsT p_{a_2 \epsT}(q, p, \theta) + a_2 \epsB p_{a_2 \epsB}(q, p, \theta)\\
			&+ a_3 A_{\texttt{WL}} p_{a_3A_{\texttt{WL}}}(q, p, \theta) + a_3A_{\texttt{WR}} p_{a_3 A_{\texttt{WR}}}(q, p, \theta)\\
			&+ a_3 \epsT p_{a_3 \epsT}(q, p, \theta) + a_3 \epsB p_{a_3 \epsB}(q, p, \theta)\\
			&+ A_{\texttt{WL}}A_{\texttt{WL}} p_{A_{\texttt{WL}} A_{\texttt{WL}}}(q, p, \theta) + A_{\texttt{WL}} A_{\texttt{WR}} p_{A_{\texttt{WL}} A_{\texttt{WR}}}(q, p, \theta)\\
			&+ A_{\texttt{WR}} A_{\texttt{WR}} p_{A_{\texttt{WR}} A_{\texttt{WR}}}(q, p, \theta)\\
			&+ A_{\texttt{WL}} \epsT p_{A_{\texttt{WL}} \epsT}(q, p, \theta) + A_{\texttt{WL}} \epsB p_{A_{\texttt{WL}} \epsB}(q, p, \theta)\\
			&+ A_{\texttt{WR}} \epsT p_{A_{\texttt{WR}} \epsT}(q, p, \theta) + A_{\texttt{WR}} \epsB p_{A_{\texttt{WR}} \epsB}(q, p, \theta)\\
			&+ \epsT \epsT p_{\epsT \epsT}(q, p, \theta) +  \epsT \epsB p_{\epsT \epsB}(q, p, \theta) +  \epsB \epsB p_{\epsB \epsB}(q, p, \theta)\,,
	\end{align*}
up to a higher order error, where the polynomial $p_{\sigma^{\alpha}}$ is the partial derivative $\left. \frac{\partial^{\alpha} RC}{\partial \sigma^{\alpha}} \right|_{\sigma = 0}$ for a multi-index $\alpha$.  Detailed fomulae for these terms are given in the Appendix.

	Next, we compute the linearization of $RC$ near the point $(q,p,\theta) = (\frac{1}{2}, \frac{1}{2}, 0)$:
	\begin{align*}
		\left.\frac{\partial RC}{\partial q}\right|_{\sigma = 0} &= 48\big (q - \frac{1}{2}\big) - 16 \big(p - \frac{1}{2}\big) + \frac{4}{3}\theta,\\
		\left.\frac{\partial RC}{\partial p}\right|_{\sigma = 0} &= -16 \big(q - \frac{1}{2}\big) +48 \big(p - \frac{1}{2}\big) + \frac{4}{3}\theta,\\
		\left.\frac{\partial RC}{\partial \theta}\right|_{\sigma = 0} &= \frac{4}{3} \big(q - \frac{1}{2}\big) + \frac{4}{3} \big(p - \frac{1}{2}\big) + \frac{7}{9}\theta.
	\end{align*}
	The Jacobian of $RC$ at $\sigma = 0$ is thus
	$$
	J = \begin{pmatrix} 48 & -16 & \frac{4}{3}\\ -16 & 48 & \frac{4}{3} \\ \frac{4}{3} & \frac{4}{3}  & \frac{7}{9}  \end{pmatrix}.
	$$
Here, and below, we have kept the $0$-coefficient terms solely as a place keeper to demonstrate that we have actually computed the coefficients of all the terms in our expansion, as well as to make it easier to verify the formulae for the interested reader.
	Next, we explore the other pieces of the linearization of $RC$, namely all first-order terms in the variables, together with the parameters.
	{\allowdisplaybreaks
	Explicitly, we have
	\begin{align*}
	%partial w.r.t. q
		\frac{\partial RC}{\partial q} 	&= \left( 48 - 112 a_1 - 112 a_2 + 112 a_3 - 256 A_{\texttt{WL}} - 256 A_{\texttt{WR}} - \frac{200}{3} \epsilon_b + \frac{104}{3} \epsilon_t \right)\left( q - \frac{1}{2} \right)\\
		&\quad  + \left(-16 + 80 a_1 + 80a_2 - 80a_3 + 40 \epsilon_b - 40\epsilon_t \right)\left( p - \frac{1}{2} \right)\\
		&\quad  + \left(\frac{4}{3} - \frac{32}{3}A_{\texttt{WL}} - \frac{32}{3}A_{\texttt{WR}} + \frac{8}{9} \epsilon_b - \frac{8}{9}\epsilon_t \right)\theta\\
		&\quad  + \Big(8a_1 -8a_2 - 8a_3 + 32 A_{\texttt{WL}} - 32 A_{\texttt{WR}} - 32 a_1 a_1 + 0 a_1a_2 + 32 a_1a_3\\
		&\quad\quad +32 a_2a_2 + 0 a_2a_3 - 32 a_3a_3 - 128a_1 A_{\texttt{WL}} + 0 a_1 A_{\texttt{WR}} + \frac{8}{3} a_1 \epsilon_t - \frac{8}{3} a_1\epsilon_b\\
		&\quad \quad  + 0a_2A_{\texttt{WL}} + 128 a_2 A_{\texttt{WR}} - \frac{8}{3} a_2 \epsilon_t + \frac{8}{3} a_2 \epsilon_b\\
		&\quad \quad  + 64 a_3 A_{\texttt{WL}} - 64 a_3 A_{\texttt{WR}} - 8 a_3 \epsilon_t + 8 a_3 \epsilon_b\\
		&\quad \quad  - 512 A_{\texttt{WL}} A_{\texttt{WL}} + 0 A_{\texttt{WL}}A_{\texttt{WR}} + 512 A_{\texttt{WR}}A_{\texttt{WR}} \\
		&\quad \quad  + \frac{32}{3} A_{\texttt{WL}} \epsilon_t - \frac{32}{3} A_{\texttt{WL}} \epsilon_b - \frac{32}{3} A_{\texttt{WR}} \epsilon_t + \frac{32}{3} A_{\texttt{WR}}\epsilon_b + 0 \epsilon_t \epsilon_t + 0 \epsilon_t \epsilon_b + 0 \epsilon_b \epsilon_b \Big)\\
		&\quad + \text{ cubic terms},
	\end{align*}
	\begin{align*}
	%partial w.r.t. p
		\frac{\partial RC}{\partial p} 	&= \left(-16 + 80 a_1 + 80a_2 - 80a_3 + 40 \epsilon_b - 40\epsilon_t \right)\left( q - \frac{1}{2} \right)\\
		&\quad + \left(48 - 112 a_1 - 112 a_2 +112 a_3 - 256 A_{\texttt{WL}} - 256 A_{\texttt{WR}} - \frac{104}{3} \epsilon_b + \frac{200}{3} \epsilon_t \right)\left( p - \frac{1}{2} \right)\\
		&\quad  + \left(\frac{4}{3} - \frac{32}{3}A_{\texttt{WL}} - \frac{32}{3}A_{\texttt{WR}} + \frac{8}{9} \epsilon_b - \frac{8}{9}\epsilon_t \right)\theta\\
		&\quad  + \left(-8a_1 + 8a_2 + 8a_3 + 32A_{\texttt{WL}} - 32 A_{\texttt{WR}} +32 a_1 a_1 + 0 a_1a_2 - 32 a_1a_3\right.\\
		&\quad\quad \left.-32 a_2a_2 + 0 a_2a_3 + 32 a_3a_3 - 64 a_1 A_{\texttt{WL}} + 64 a_1 A_{\texttt{WR}} - 8 a_1 \epsilon_t + 8 a_1\epsilon_b\right.\\
		&\quad \quad  \left. - 64 a_2A_{\texttt{WL}} + 64 a_2 A_{\texttt{WR}} + 8 a_2 \epsilon_t - 8 a_2 \epsilon_b\right.\\
		&\quad \quad \left. + 0 a_3 A_{\texttt{WL}} - 128 a_3 A_{\texttt{WR}} + \frac{8}{3} a_3 \epsilon_t - \frac{8}{3} a_3 \epsilon_b\right.\\
		&\quad \quad \left. - 512 A_{\texttt{WL}} A_{\texttt{WL}} + 0 A_{\texttt{WL}}A_{\texttt{WR}} + 512 A_{\texttt{WR}}A_{\texttt{WR}} \right.\\
		&\quad \quad \left. + \frac{32}{3} A_{\texttt{WL}} \epsilon_t - \frac{32}{3} A_{\texttt{WL}} \epsilon_b + \frac{32}{3} A_{\texttt{WR}} \epsilon_t + \frac{32}{3} A_{\texttt{WR}}\epsilon_b + 0 \epsilon_t \epsilon_t + 0 \epsilon_t \epsilon_b + 0 \epsilon_b \epsilon_b \right)\\
		&\quad + \text{ cubic terms},
	\end{align*}
	and
	\begin{align*}
	%partial w.r.t. theta
		\frac{\partial RC}{\partial \theta} 	&= \left(\frac{4}{3} - \frac{32}{3}A_{\texttt{WL}} - \frac{32}{3}A_{\texttt{WR}} + \frac{8}{9} \epsilon_b - \frac{8}{9}\epsilon_t \right)\left( q - \frac{1}{2} \right)\\
		&\quad  + \left(\frac{4}{3} - \frac{32}{3}A_{\texttt{WL}} - \frac{32}{3}A_{\texttt{WR}} + \frac{8}{9} \epsilon_b - \frac{8}{9}\epsilon_t \right)\left( p - \frac{1}{2} \right)\\
		&\quad  + \left(\frac{7}{9} - \frac{5}{9} a_1 - \frac{5}{9}a_2 + \frac{5}{9} a_3 - \frac{32}{9}A_{\texttt{WL}} - \frac{32}{9}A_{\texttt{WR}} + \frac{1}{54}\epsilon_b - \frac{1}{54} \epsilon_t \right)\theta\\
		&\quad + \left(\frac{2}{3}a_1 - \frac{2}{3}a_2 + \frac{2}{3} a_3 + \frac{8}{3}A_{\texttt{WL}} - \frac{8}{3}A_{\texttt{WR}} - \frac{4}{3} a_1 a_1 + 0 a_1a_2 + 0 a_1a_3\right.\\
		&\quad\quad \left.+ \frac{4}{3} a_2a_2 - \frac{4}{3} a_2a_3 + \frac{4}{3} a_3a_3 - 8 a_1 A_{\texttt{WL}} - \frac{8}{3} a_1 A_{\texttt{WR}} - \frac{1}{9} a_1 \epsilon_t + \frac{1}{9} a_1\epsilon_b\right.\\
		&\quad \quad  \left.+ \frac{8}{3} a_2A_{\texttt{WL}} + 8 a_2 A_{\texttt{WR}} + \frac{1}{9} a_2 \epsilon_t - \frac{1}{9} a_2 \epsilon_b\right.\\
		&\quad \quad \left. - \frac{8}{3} a_3 A_{\texttt{WL}} - 8 a_3 A_{\texttt{WR}} - \frac{1}{9} a_3 \epsilon_t + \frac{1}{9} a_3 \epsilon_b\right.\\
		&\quad \quad \left. - \frac{128}{3} A_{\texttt{WL}} A_{\texttt{WL}} + 0 A_{\texttt{WL}}A_{\texttt{WR}} + \frac{128}{3} A_{\texttt{WR}}A_{\texttt{WR}} \right.\\
		&\quad \quad \left. - \frac{4}{9} A_{\texttt{WL}} \epsilon_t + \frac{4}{9} A_{\texttt{WL}} \epsilon_b + \frac{4}{9} A_{\texttt{WR}} \epsilon_t - \frac{4}{9} A_{\texttt{WR}}\epsilon_b + 0 \epsilon_t \epsilon_t + 0 \epsilon_t \epsilon_b + 0 \epsilon_b \epsilon_b \right)\\
		&\quad + \text{ cubic terms}.
	\end{align*}
	With these expansions, we can express the linearization as $J+ J_{\sigma}$, where the terms $J_{\sigma,ii}$ come from the partials computed above. Explicitly,
	\begin{align*}
		J_{\sigma,11} 	&= - 112 a_1 - 112 a_2 - 256 A_{\texttt{WL}} - 256 A_{\texttt{WR}} - \frac{200}{3} \epsilon_b + \frac{104}{3} \epsilon_t,\\
		J_{\sigma,12}	&= 80 a_1 + 80a_2 + 40 \epsilon_b - 40\epsilon_t,\\
		J_{\sigma,13} 	&= - \frac{32}{3}A_{\texttt{WL}} - \frac{32}{3}A_{\texttt{WR}} + \frac{8}{9} \epsilon_b - \frac{8}{9}\epsilon_t,\\
		J_{\sigma,21} 	&= 80 a_1 + 80a_2 + 40 \epsilon_b - 40\epsilon_t,\\
		J_{\sigma,22} 	&= - 112 a_1 - 112 a_2 - 256 A_{\texttt{WL}} - 256 A_{\texttt{WR}} - \frac{104}{3} \epsilon_b + \frac{200}{3} \epsilon_t,\\
		J_{\sigma,23} 	&= - \frac{32}{3}A_{\texttt{WL}} - \frac{32}{3}A_{\texttt{WR}} + \frac{8}{9} \epsilon_b - \frac{8}{9}\epsilon_t,\\
		J_{\sigma,31} 	&= - \frac{32}{3}A_{\texttt{WL}} - \frac{32}{3}A_{\texttt{WR}} + \frac{8}{9} \epsilon_b - \frac{8}{9}\epsilon_t,\\
		J_{\sigma,32} 	&= - \frac{32}{3}A_{\texttt{WL}} - \frac{32}{3}A_{\texttt{WR}} + \frac{8}{9} \epsilon_b - \frac{8}{9}\epsilon_t,\\
		J_{\sigma,33} 	&= - \frac{5}{9} a_1 - \frac{5}{9}a_2 - \frac{32}{9}A_{\texttt{WL}} - \frac{32}{9}A_{\texttt{WR}} + \frac{1}{54}\epsilon_b - \frac{1}{54} \epsilon_t.
	\end{align*}

	With these pieces we can express the criterion for being a critical point of the Ratio Cut, incorporating linear terms in the parameters, as
	\begin{align*}
		(J+J_{\sigma}) \begin{pmatrix} q - \frac{1}{2}\\ p - \frac{1}{2}\\ \theta \end{pmatrix} &= L(\sigma)\\
		&:= a_1 \begin{pmatrix} 8 \\ -8 \\ \frac{2}{3}\end{pmatrix} + a_2 \begin{pmatrix}-8\\ 8\\ -\frac{2}{3}\end{pmatrix} + a_3 \pmat{-8 \\ 8 \\ \frac{2}{3}} + A_{\texttt{WL}} \begin{pmatrix}32 \\ 32 \\ \frac{8}{3} \end{pmatrix} + A_{\texttt{WR}} \begin{pmatrix}-32\\ -32 \\ -\frac{8}{3}\end{pmatrix}\\
		&\quad  + a_1a_1 \pmat{-32 \\ 32 \\ - \frac{4}{3}} + a_1a_2 \pmat{0 \\ 0 \\ 0} + a_1a_3 \pmat{32 \\ -32 \\ 0} \\
		&\quad  + a_2a_2 \pmat{32 \\ -32 \\ \frac{4}{3}} + a_2a_3 \pmat{0 \\ 0 \\ -\frac{4}{3}} + a_3a_3 \pmat{ - 32 \\ 32 \\ \frac{4}{3} } \\
		&\quad 	+ a_1A_{\texttt{WL}} \pmat{ -128 \\ -64 \\ -8 } + a_1 A_{\texttt{WR}} \pmat{ 0 \\ 64 \\ - \frac{8}{3} } + a_1 \epsilon_t \pmat{ \frac{8}{3} \\ -8 \\ - \frac{1}{9} } + a_1 \epsilon_b \pmat{ - \frac{8}{3} \\ 8 \\ \frac{1}{9} } \\
		&\quad 	+ a_2 A_{\texttt{WL}} \pmat{ 0 \\ -64 \\ \frac{8}{3} } + a_2 A_{\texttt{WR}} \pmat{ 128 \\ 64 \\ 8 } + a_2 \epsilon_t \pmat{ - \frac{8}{3} \\ 8 \\ \frac{1}{9} } + a_2 \epsilon_b \pmat{ \frac{8}{3} \\ - 8 \\ -\frac{1}{9} } \\
		&\quad  + a_3 A_{\texttt{WL}} \pmat{ 64 \\ 0 \\ -\frac{8}{3} } + a_3 A_{\texttt{WR}} \pmat{ -64 \\ - 128 \\ -8 } + a_3 \epsilon_t \pmat{ -8 \\ \frac{8}{3} \\ -\frac{1}{9} } + a_3 \epsilon_b \pmat{ 8 \\ -\frac{8}{3} \\ \frac{1}{9} } \\
		&\quad  + A_{\texttt{WL}} A_{\texttt{WL}} \pmat{ -512 \\ -512 \\ -\frac{128}{3}  } + A_{\texttt{WL}} A_{\texttt{WR}} \pmat{ 0 \\ 0 \\ 0 } + A_{\texttt{WR}} A_{\texttt{WR}} \pmat{ 512 \\ 512 \\ \frac{128}{3}  } \\
		&\quad  + A_{\texttt{WL}} \epsilon_t \pmat{\frac{32}{3} \\ \frac{32}{3} \\ - \frac{4}{9}  } + A_{\texttt{WL}} \epsilon_b \pmat{ - \frac{32}{3} \\ - \frac{32}{3} \\ \frac{4}{9} } \\
		&\quad  + A_{\texttt{WR}} \epsilon_t \pmat{ - \frac{32}{3} \\ \frac{32}{3} \\ \frac{4}{9} } + A_{\texttt{WR}} \epsilon_b \pmat{ \frac{32}{3} \\ \frac{32}{3} \\ -\frac{4}{9}  } \\
		&\quad  + \epsilon_t \epsilon_t \pmat{0\\0\\0   } + \epsilon_t \epsilon_b \pmat{0\\ 0\\ 0  } + \epsilon_b \epsilon_b \pmat{0\\ 0 \\ 0  }\\
		&\quad + \text{ cubic terms in the parameters}.
	\end{align*}
	}
	\subsection*{Using the Implicit Function Theorem}
	
	Suppose $v$ is a critical triple of values for the Ratio Cut, i.e. $(J+J_{\sigma})v = 0$. Then near $\sigma = 0$, we can solve for $v$ in terms of the parameters in $\sigma$:
	\begin{align*}
		v 	&= J^{-1} L(\sigma)\\
			&= \pmat{\frac{1}{12}(a_1 - a_2 - 2a_3) + (A_{\texttt{WL}} - A_{\texttt{WR}}) \\ -\frac{1}{12} (2a_1 - 2a_2 - a_3) + (A_{\texttt{WL}} - A_{\texttt{WR}}) \\ a_1 - a_2 + a_3} + \text{quadratic terms in the parameters}.
	\end{align*}
	We will not write out the full $J^{-1}L(\sigma)$ here, though it is used for the approximations later and its form becomes clear in the representations below.

	To summarize, we have shown in this section that, near the critical point $(p,q,\theta) = \left( \frac{1}{2}, \frac{1}{2}, 0 \right)$ of the ratio cut, and for $\sigma=(a_1, a_2, \epsilon_b, \epsilon_t, A_{\texttt{WL}}, A_{\texttt{WR}})$ near $0$, we have 
	\begin{align*}
		\begin{pmatrix} q - \frac{1}{2} \\ p - \frac{1}{2} \\ \theta \end{pmatrix} = \pmat{\frac{1}{12}(a_1 - a_2 - 2a_3) + (A_{\texttt{WL}} - A_{\texttt{WR}}) \\ -\frac{1}{12} (2a_1 - 2a_2 - a_3) + (A_{\texttt{WL}} - A_{\texttt{WR}}) \\ a_1 - a_2 + a_3}+ \text{quadratic terms in the parameters}.
	\end{align*}

	\section{Are our domains generic?}
	
	Since Ratio Cuts are circular, it seems more natural to prove Theorem 2 for circular, instead of parabolic, boundary curves. Using circular arcs, however, proved to be intractable for our computational methods, whereas parabolic arcs could be used. This section shows that such approximations only incur third order errors, and thus Theorem 2 is still applicable for circular boundary curves. This first lemma shows this approximation is valid near the intersection of the cut and the original domain's boundary. 
	
	\begin{lemma}
	Provided $I(Q)$ defined in \eqref{Definition I(Q)} is sufficiently small, we may approximate the top and bottom curves by parabolas in a neighborhood of the Ratio Cut.
	\end{lemma}
	
	\begin{proof}
	The small $I(Q)$ makes sure the domain is approximately trapezoidal with a, for example, $1:2$, aspect ratio, and that the top and curves are $C^3$ with small oscillations on a uniform spatial scale.  Hence, Taylor expanding the top and bottom curves out to $2$nd order on this scale near $p=1/2, q = 1/2$, we can create a parabolic approximation that is accurate up to $O( |p-\frac12|^3, |q-\frac12|^3)$ in this region. Potentially modifying the left and right wings in order to ensure the wings of domain have the appropriate error outside this uniform neighborhood, we see that these curvilinear trapezoids will have ratio cuts that are indeed approximated up to lower order terms in the same fashion as our exact parabolic trapezoid calculations.   
	\end{proof}
		
	In practice, we are given a domain $Q$ whose top and bottom boundary curves are more general to start, and upon iteration of our domains we are most interested when sides are given by circular arcs. To handle smooth enough domains of this form in full generality, we can construct a parabolic trapezoid by: rescaling the domain to fit our aspect ratio; locally using a parabolic approximation for the top and bottom curves; and finally cutting off left and right portions of $\Omega$, which are treated as black-box regions. The Ratio Cut function sees these regions as general wing areas, denoted $A_{\texttt{WL}}$ and $A_{\texttt{WR}}$ for the left- and right-wing areas respectively.  The measure of distance from the rectangle, $I(Q)$, in which we measure the deflection of our domains, allows us to approximate our quadrilateral by a parabolic curve with errors that are higher order in the parameter space. 
		
	\subsection{Can a circular arc be parabolically approximated?}
	
	Optimal cuts for generic parabolic trapezoids are circular. Because of this, we would expect the top and bottom boundary curves of the domains in the next iteration we have been considering to be circular, but not parabolic. Here we show that parabolic arcs approximate circular arcs well, and that the ratio cut for circular boundary curves differs from the ratio cut for parabolic boundary curves at third order and higher. Thus, working with parabolic curves does not incur any significant error in the ratio cut series expansion.
	
	\begin{lemma}
	   Up to higher order error terms in the curve parameters, we may approximate the top and bottom circular arc curves as paraboloids encompassing the correct area.
	   \end{lemma}
	   
	   The remainder of this section will be devoted to the proof of this Lemma.
	
	\subsubsection{Approximating Circular Arcs} 
	First let us consider the most important example for our iterated domain conjecture, in which the top and bottom boundaries are circular arcs.  In what follows, the top and bottom parabolic curves will be, respectively,
	\begin{align*}
		y_T^P(x) 	&= \epsilon_t x^2 + (a_1 - a_2 - \epsilon_t)x+ \left( a_1 + \frac{1}{2} \right),\\
		y_B^P(x) 	&= \epsilon_b x^2 + (a_3 - \epsilon_b)x.
	\end{align*}
	
	The parameters $a_1, a_2, a_3$ denote horizontal perturbations from the top-left, top-right, and bottom-right vertices of a rectangle, so our parabolic trapezoid has vertices $(0,0)$, $\left(0, \frac{1}{2}+a_1 \right)$, $\left(1, \frac{1}{2}+a_2 \right),$ and $\left(1, a_3 \right)$. The terms $\epsilon_t$ and $\epsilon_b$ are curvature parameters that specify the shape of the two boundary curves.
	
	The circular arcs are given by the formulas
	\begin{align*}
		y_T^C(x) &= c_{y,t} + \sqrt{r_t^2 - (x-c_{x,t})^2}, \\
		y_B^C(x) 	&= c_{y,b} - \sqrt{r_b^2 - (x-c_{x,b})^2},
	\end{align*}
	where $r_i$ and $(c_{x,i},c_{y,i})$ are the radius and center of the corresponding top ($i=T$) or bottom ($i=B$) circle:
	
	\begin{align*}
		r_t &= \frac{\sqrt{1 + (a_1-a_2)^2}}{2 \sin(\frac{\theta}{2})},\\
		r_b &= \frac{\sqrt{1+a_3^2}}{2 \sin(\frac{\theta}{2})},
	\end{align*}
	and the center point coordinates are found from the systems of equations (for $i=T$ and $i=B$)
	\begin{align*}
		&\begin{cases}
		(c_{x,t})^2 + (\frac{1}{2} + a_1 - c_{y,t})^2 	= r_t^2,\\
		(1-c_{x,t})^2 + (\frac{1}{2}+a_2 - c_{y,t})^2 	= r_t^2,  
		\end{cases}\\
		\text{and } 
		&\begin{cases}
		(c_{x,b})^2 + (c_{y,b})^2 	= r_b^2,\\
		(1-c_{x,b})^2 + (a_3 - c_{y,b})^2 	= r_b^2.
		\end{cases}
	\end{align*}
	
	As an explicit example, expressing $y_B^C$ in terms of the parameters gives
	\begin{align*}
	%yBC
		y_B^C(x) 	&= \frac{a_3}{2} + \cot\left( \frac{\theta_b}{2} \right) - \left( \frac{1}{1+a_3^2} \left( a_3^4 - (1-2x)^2 - 4a_3^2 (x-1) x \right. \right. \\
		&\quad \left.\left. + 2a_3 (1+a_3^2) (1-2x) \cot\left( \frac{\theta_b}{2} \right) + (1+a_3^2) \csc^2\left( \frac{\theta_b}{2} \right)  \right) \right)^{\frac{1}{2}}.
	\end{align*}
	
	The next lemma shows how to choose $\epsilon_t$ and $\epsilon_b$ so that the parabolic trapezoid's area approximates the circular trapezoid's area up to third order. These curvature terms are chosen so that the quadrilateral's parabolic caps have, up to third order terms, the same area as corresponding circular arcs.
	
	\begin{lemma}
	Let $l_T(x) = (1-t)\left(\frac{1}{2}+a_1 \right) + t \left(\frac{1}{2}+a_2 \right)$, and write 
	\begin{equation}
	S_T^C = \{ (x,y) \colon l_T(x) \leq y \leq y_T^C(x) \}, \,\, S_T^P(x) = \{ (x,y) \colon l_T(x) \leq y \leq y_T^P(x) \}
	\end{equation} 
	for the regions bounded below by the straight line $\{\{ t, l_T(x) \} \colon t \in [0,1] \}$ and above by the curves $\{\{t, y_T^i(x) \} \colon x\in[0,1], i = C \text{ or } P \}$ respectively. Similarly, let $l_B(x) = t a_3,$ and write
	 \begin{equation}
	S_B^C = \{ (x,y) \colon l_B(x) \leq y \leq y_B^C(x) \},\,\, S_B^P(x) = \{ (x,y) \colon l_B(x) \leq y \leq y_B^P(x) \}
	\end{equation} 
	for the regions bounded above by the straight line $\{\{ t, l_B(x) \} \colon t \in [0,1] \}$ and below by the curves $\{\{t, y_B^i(x) \} \colon x\in[0,1], i = C \text{ or } P\}$ respectively.
	
		For $\epsilon_t = -\frac{1+(a_1-a_2)^2}{2}\theta_t$ and $\epsilon_b = \frac{1+a_3^2}{2}\theta_b$, we have
		\begin{align*}
			|S_T^C| -|S_T^P| 	&= O(\theta_t^3)
			\end{align*}
			and
			\begin{align*}
						|S_B^C| - |S_B^P| 	&= O(\theta_b^3) .
		\end{align*}
	\end{lemma}
	
	\begin{proof}
		Both $|S_T^C|$ and $|S_B^C|$ can be found using basic trigonometry: the area of a circular segment with radius $r$ and angle $\theta$ is $\frac{1}{2} r^2 (\theta - \sin(\theta))$. In our case,
		\begin{align*}
			|S_T^C| 	&= \frac{1}{2} \left(\frac{\sqrt{1+(a_1-a_2)^2}}{2\sin(\frac{\theta_t}{2})} \right)^2 (\theta_t - \sin(\theta_t) )\\
				&= \frac{1+(a_1-a_2)^2}{12}\theta_t + \frac{1+(a_1-a_2)^2}{360}\theta_t^3 + O(\theta_t^4),\\
			|S_B^C| 	&= \frac{1}{2} \left(\frac{\sqrt{1+a_3^2}}{2\sin(\frac{\theta_b}{2})} \right)^2(\theta_b - \sin(\theta_b))\\
				&= \frac{1+a_3^2}{12} \theta_b + \frac{1+a_3^2}{360} \theta_b^3 + O(\theta_b^4).\\
		\end{align*}
		
		For $|S_T^P|$ and $|S_B^P|$, we integrate:
		\begin{align*}
			|S_T^P| 	&= \int_{0}^1 y_T^P(x) - l_T(x)dx\\
				&= \int_0^1 \epsilon_t x^2 + (a_2 - a_1 - \epsilon_t)x + \left( \frac{1}{2} + a_1 \right) - (1-x)\left( \frac{1}{2}+a_1 \right) - x\left(\frac{1}{2}+a_2 \right) dx\\
				&= -\frac{\epsilon_t}{6}
				\end{align*}
		and 
		\begin{align*}
		 |S_B^C| 	&= \int_0^1 l_B(x) - y_B^C(x) dx\\
			&= \int_0^1 a_3 x -\epsilon_b x^2 - (a_3 - \epsilon_b)x dx
			= \frac{\epsilon_b}{6}.
		\end{align*}
		We want to choose constants $C_t$ and $C_b$ so that setting $\epsilon_t = C_t \theta_t$ and $\epsilon_b = C_b \theta_b$ gives us $|S_T^C|-|S_T^P| = O(\theta_t^3)$ and $|S_B^C| - |S_B^P| = O(\theta_b^3)$. Indeed, letting $\epsilon_t = -\frac{1+(a_1-a_2)^2}{2}\theta_t$ and $\epsilon_b = \frac{1+a_3^2}{2}\theta_b$ gives us this result.
	\end{proof}
	
	Since the parabolic and circular trapezoids only differ in the type of caps on each, making these choices for $\epsilon_t$ and $\epsilon_b$ ensures the area of the circular and parabolic trapezoids are very well approximated. In addition, the boundary curves agree up to third order, and so the left areas and ratio cut lengths also agree to at least third order (in the parameters).
	\begin{cor} 
		Setting $\epsilon_t = -\frac{1+(a_1-a_2)^2}{2}\theta_t$ and $\epsilon_b = \frac{1+a_3^2}{2}\theta_b$, we get the following approximations:
		\begin{align*}
		y_T^C(x) - y_T^P(x) &= \frac{1}{48}(1+(a_1-a_2)^2) \theta_t^2 \left( O(a_1) + O(a_2) + O(\theta_t)  \right),\\
		y_B^C(x) - y_B^P(x) &= \frac{1}{48} (1+a_3^2) \theta_b^2 (O(a_3) + O(\theta_b)).
		\end{align*}
		
		Moreover, we have
		\begin{align*}
		A_T^C 	&= A_T^P + O(\theta_t^3) + O(\theta_b^3),\\
		\text{and } A_L^C 	&= A_L^P + O(\theta_t^2)(O(a_1) + O(a_2)) + O( \theta_b^2)O(a_3),
		\end{align*}
		where $A_T^i$, $A_L^i$ is the total area and left area respectively of the parabolic ($i=P$) or circular ($i=C$) trapezoid, and
		\begin{align*}
			\|(p,y_T^C(p)) - (q,y_B^C(q)) \| 	&= \|(p,y_T^P(p)) - (q,y_B^P(q))\| + O(\theta_t^3)+O(\theta_b^3).
		\end{align*}
	\end{cor}
	
	The corollary justifies our use of only quadratic terms in the Ratio Cut Taylor expansion, since the parabolic and circular Ratio Cuts (i.e. Ratio Cut functions of parabolic or circular trapezoids) agree up to third order in the parameters of the domains, given the proper curvature term substitutions.
	
	\begin{proof}
		The parabolic and circular boundary curve approximations come by Taylor expansions.
		For the areas, recall that the parabolic and circular trapezoids differ only in their ``caps''. Thus, assuming $a_1<a_2$ and $a_3<0$, the total area of the circular trapezoid is
		\begin{align*}
			A_T^C 	&= \left(\frac{1}{2} + a_1\right) + \frac{1}{2} \left(a_2-a_1 \right) + \frac{1}{2}\left(-a_3 \right) + \area(S_T^C) + \area(S_B^C)\\
				&= \left(\frac{1}{2} + a_1\right) + \frac{1}{2} \left(a_2-a_1 \right) + \frac{1}{2}\left(-a_3 \right) + \area(S_T^P) + \area(S_B^P) + O(\theta_t^3) + O(\theta_b^3)\\
				&= A_T^P + O(\theta_t^3) + O(\theta_b^3).
		\end{align*}
		Write 
		\begin{equation}
		L^C := \|(p,y_T^C(p)) - (q,y_B^C(q))\|,\,\,L^P := \|(p,y_T^P(p)) - (q,y_B^P(q))\|; 
		\end{equation}
		these quantities are the lengths of the line segments connecting the intersection points of intersection between the ratio cut and boundary of the region. Expanding $L^C$ into a Taylor series in $\theta_b$ and $\theta_t$ gives 
		\begin{align*}
			L^C = L^P + O(\theta_b^3) + O(\theta_t^3).
		\end{align*}
		
		Finally, write 
		\begin{equation}
		\bar{y}_j^P(x) := y_j^P(x) + r_j(x), 
		\end{equation}
		where $r_j$ is quadratic in the parameters, for $j = T,B$. Let $\gamma_1$ be the left boundary segment of the (circular or parabolic) trapezoid, $\gamma_2$ the bottom boundary segment, $\gamma_3$ the right boundary segment, and finally $\gamma_4$ the top boundary segment, such that $\sum_i \gamma_i$ is oriented counter-clockwise.
		
		In general, if a function $y=f(x)$ is expressed as a sum $f(x) = g(x) + h(x)$, then in utilizing Stokes' theorem we get
		\begin{align*}
			\frac{1}{2} \int x dy - ydx 	&= \frac{1}{2} \int x d(g(x) + h(x)) - (g(x)+h(x)) dx\\
				&= \frac{1}{2} \int x dg + xdh - g dx - h dx\\
				&= \frac{1}{2} \int xdg - g dx + \frac{1}{2} \int x dh - h dx.
		\end{align*}
		In short, the line integrals involved in Stokes' theorem are linear in the input $y$.
		
		Since we are approximating the circular arcs as parabolic arcs plus correction terms, we only need to account for the pieces of the area line integral that incorporate the correction pieces. The left boundary curve is the same for both the parabolic and circular trapezoids, so splitting up the line integrals into parabolic pieces plus correction terms gives 
		\begin{align*}
			A_L^C &= \sum_i \frac{1}{2} \int_{\gamma_i} xdy - ydx + A_S^C\\
			&=\, A_L^P + \frac{1}{2} \int_{0}^{q} x d (r_b(x)) - r_b(x) dx\\
			&\quad + \frac{1}{2} \int_{q}^{p} l_{right,x}(x) d(l_{right,y}(x)) - l_{right,y}(x) d(l_{right,x}(x))\\
			&\quad + \frac{1}{2} \int x d(r_t(x)) - r_t(x) dx + O(\theta_b^3)+O(\theta_t^3),  
		\end{align*}
		where the $A_S^i$'s are the the circular segment areas bounded by the straight line connecting the two intersection points and the ratio cut, and 
		\begin{align*}
		(l_{right,x}(x),l_{right,y}(x))& \\
		:= \frac{p - x}{p - q}(p, y_T^C(p))& + \frac{x - q}{p-q}(q, y_B^C(q)) - \left( \frac{p - x}{p - q}(p, y_T^P(p)) + \frac{x - q}{p-q}(q, y_B^P(q)) \right)
		\end{align*} 
		is the error between the line segments connecting the parabolic arc intersection points and the circular arc intersection points. 
		
		Since the functions $r_i(x), l_{right,x}(x), l_{right,y}(x)$ are at least third order in the parameters, the line integrals involving these terms also agree up to third order. Finally, as indicated in the computation, the circular segment area $A_S^C$ is cubic in the parameters because the lengths of the line segments agree up to cubic terms. Thus, $A_L^C = A_L^P + O(\theta_b^2) O(a_3)+O(\theta_t^2)( O(a_1) + O(a_2))$. 
	\end{proof}
	
	Putting the above results together, we conclude that the ratio cuts corresponding to circular and parabolic boundary curves differ only in cubic terms of the parameters. Putting these pieces together, we see the same is true for the full ratio cut.
	
	\begin{lemma}
		Let $\Gamma^C, A_T^C,$ and $A_L^C$ be the ratio cut, total area, and left-area respectively of a domain constructed from circular boundary arcs. Then we can construct a new domain with parabolic arcs such that
		\begin{align*}
			RC^C - RC^P 	&= \frac{|\Gamma^C|}{(A_T^C-A_L^C)A_L^C} - \frac{|\Gamma^P|}{(A_T^P-A_L^P)A_L^P} = O(|r|^3),
		\end{align*}
		where $\Gamma^P, A_{T}^P,$ and $A_{L}^P$ are the ratio cut, total area, and left-area corresponding to a domain with parabolic boundary arcs.
	\end{lemma}
	
	\begin{proof} 
		By the previous lemma (and its corollary), we can write
		\begin{align*}
			A_T^C 	&= A_T^P + O(\theta_t^3)+O(\theta_b^3),\\
			A_L^C 	&= A_L^P + O(\theta_t^2)(O(a_1)+O(a_2)) + O(\theta_b^2)O(a_3),\\
			l^C 	&= l^P + O(\theta_b^3) + O(\theta_t^3),
		\end{align*}
		and so
		\begin{tiny}
		\begin{align*}
			RC^C 	&= \frac{|\Gamma^C|}{(A_T^C - A_L^C)A_L^C}\\
					&= \frac{ \frac{l^P + O(\theta_b^3) + O(\theta_t^3)}{\sinc(\frac{\theta}{2})}}{(A_T^P + O(\theta_t^3)+O(\theta_b^3) - A_L^P - O(\theta_t^2)(O(a_1)+O(a_2)) - O(\theta_b^2)O(a_3) )(A_L^P + O(\theta_t^2)(O(a_1)+O(a_2)) + O(\theta_b^2)O(a_3))}.
		\end{align*}
		\end{tiny}
		Taylor expanding the difference between $RC^P$ and $RC^C$ in terms of the parabolic approximations gives
		\begin{align*}
			RC^P - RC^C 	&= O(a_1\theta_t^2) + O(a_2\theta_t^2) + O(\theta_t^3) + O(a_3\theta_b^2) + O(\theta_b^3).
 		\end{align*}
 		Since the parabolic approximation incurs only third (and higher) order errors, so does the full ratio cut function as seen in the Taylor series expansion.
	\end{proof}
	
	This lemma lets us safely use the parabolic approximation when investigating the dependency of the ratio cut on the parameters.

		\section{Discussion and Conclusion}

	\subsection{Pictures comparing optimizers with approximations}
	
	First, let us demonstrate how well our approximation does to compute the Ratio Cut when the curves are parabolic arcs.  The set of plots in Figures \ref{compplot1}-\ref{compplot8} compare optimal values of $RC$ with the linear approximations in $p, q,$ and $\theta$.
	Each plot has the optimal ratio cut value (solid blue) for a family of parameters, indicated next to the $x$-axis. The circle parameters for the optimal cuts were computed numerically using Mathematica's {\texttt{FindMinimum}} command. The dot-dashed red curves correspond to the linear approximations for optimal $q, p,$ and $\theta$. The figure below each plot corresponds to an extremal configuration used to compare optimal and approximate ratio cut values. 
	All parameters $a_1, a_2, a_3, \epsilon_b, \epsilon_t, A_{\texttt{WL}}$, and $A_{\texttt{WR}}$ are zero unless otherwise stated.

		\subsection{Potential Applications}

Given domains that are "nearly rectangular" in the appropriate sense, we have shown here that spectral cuts determined by the $1$-Laplacian in geometric domains consist of area minimizing curves that reduce the distance to being rectangular through a series of qualitative geometric results and a careful asymptotic expansion in a parametrization of the cut near a point of (near) symmetry in the domain.  Given the isoperimetric nature of the cuts, we propose that $1$-Laplacian partitioning could be useful in data analysis, however we acknowledge that the numerical tools for computing the Ratio Cut would likely need to be improved computationally to make large data sets tractable.  Indeed, the isoperimetric nature of the Ratio Cut makes it clear to imagine how partitioning could occur, but also shows that approximating a continuum domain by a graph accurately requires great care due to the subtlety with which graph distances can fluctuate in point clouds.

Consistency of graph $1$-Laplacian spectral cuts was studied in \cite{trillos2016consistency}.  See \cite{Buhler_Hein:2009,Hein_Buhler:2010,BHcode} for current computational methods for computing $1$-Laplacian states using proximal-dual type numerical schemes.   Similar ideas have been proposed in \cite{bresson2013adaptive}.  See also \cite{nossek2018flows} for similar schemes originating from a diffusion-type equation, with the intent of forcing the equation to converge to a nonlinear eigenfunction.

	\begin{figure}
			\includegraphics[width=0.25\linewidth]{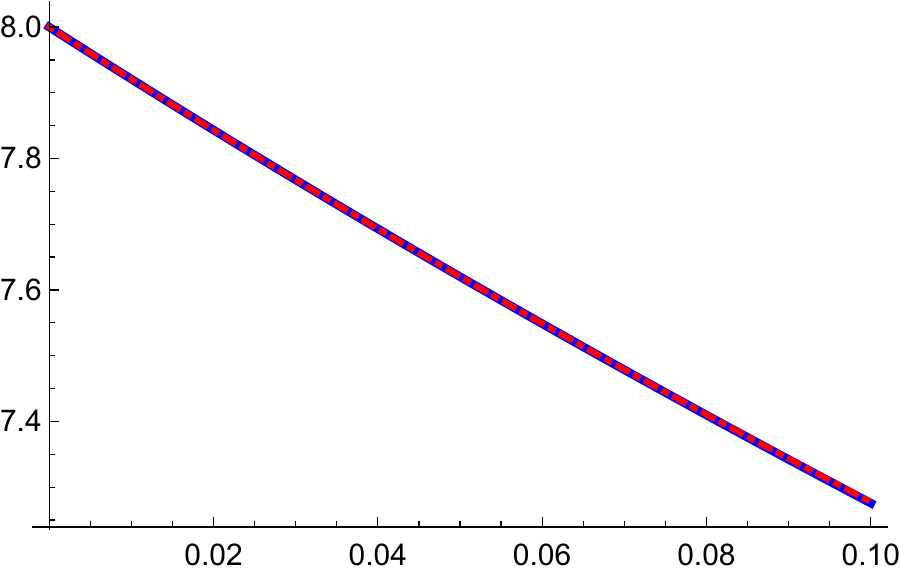}
			\includegraphics[width=0.25\linewidth]{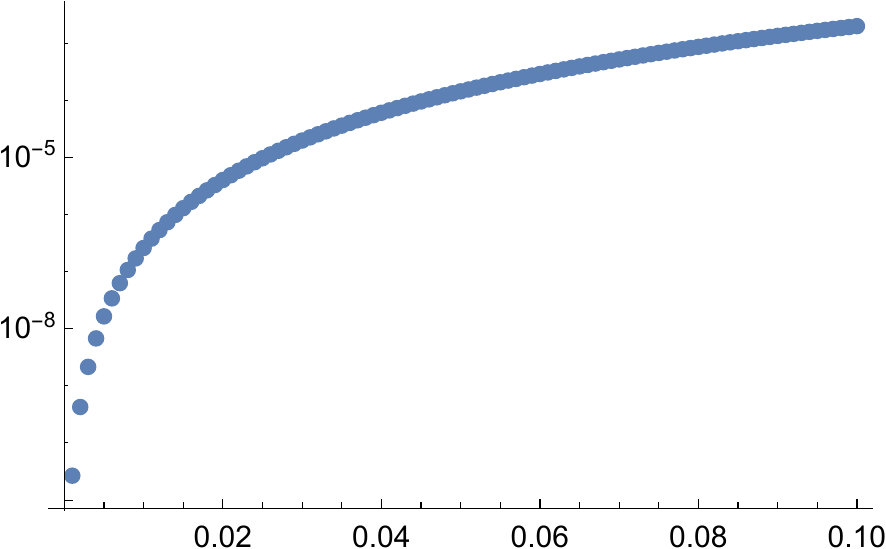}  \\
			\includegraphics[width=0.2\textwidth]{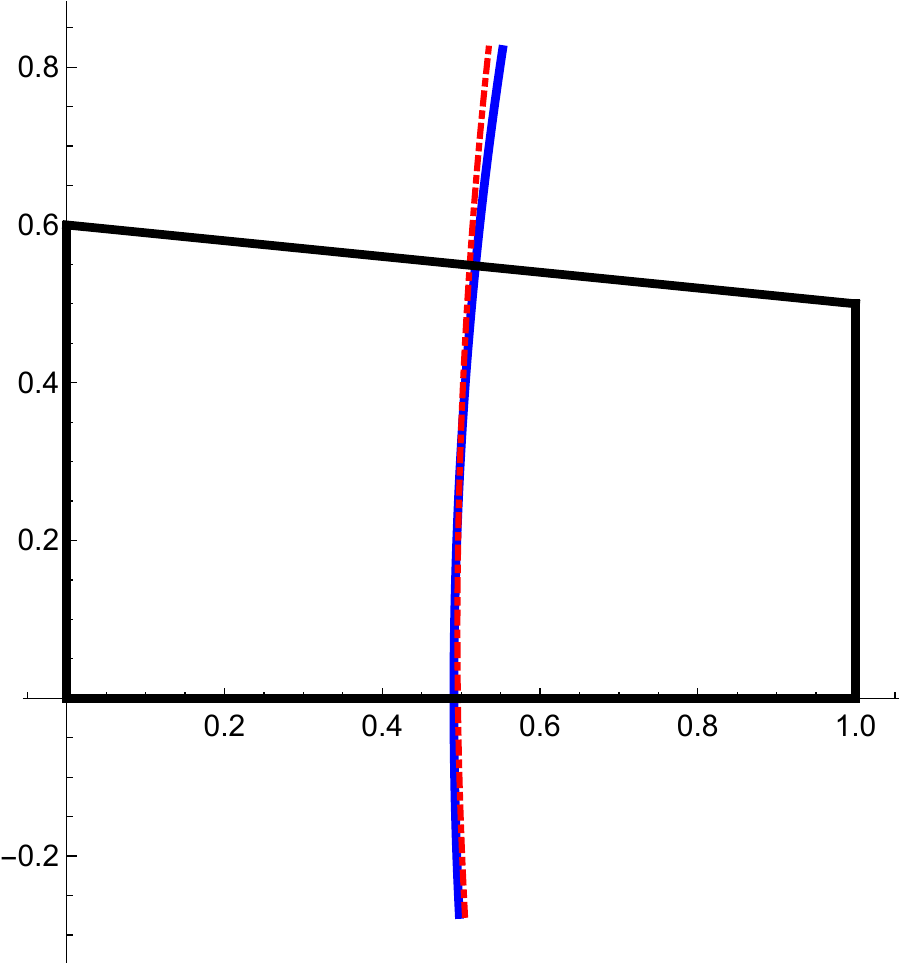}
			\caption{(Top Left) Ratio cut values corresponding to the optimal circular cuts (solid blue line) and parabolically approximated cuts (dot-dashed red line) for domains with various $a_1$ values between $0$ and $0.1$. (Top Right)  The absolute error between the approximate and optimal ratio cut values in the left plot.  (Bottom) The optimal cut (solid blue line) and parabolically approximated cut (dot-dashed red line) for the trapezoidal domain with $a_1 = 0.1$.}
				\label{compplot1}
		\end{figure}
	\begin{figure}
		\includegraphics[width=0.25\linewidth]{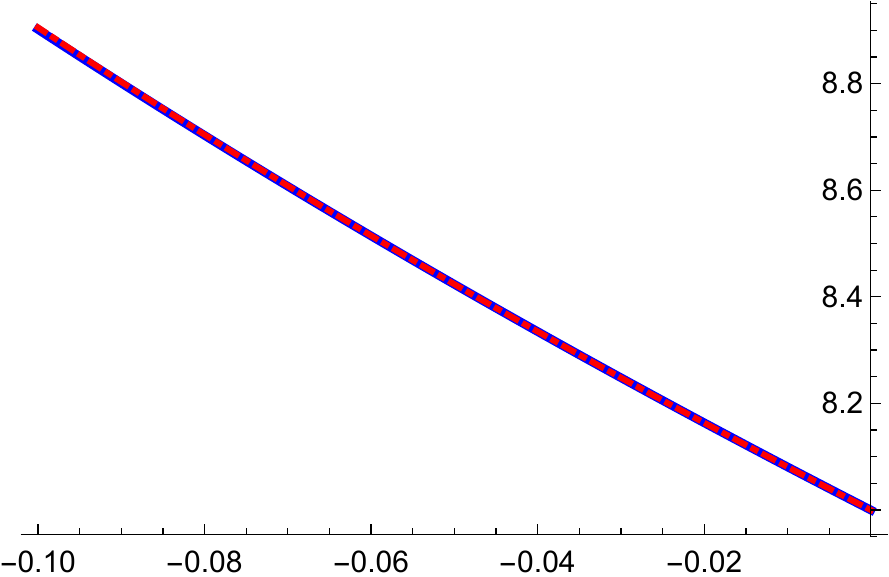}
		\includegraphics[width=0.25\linewidth]{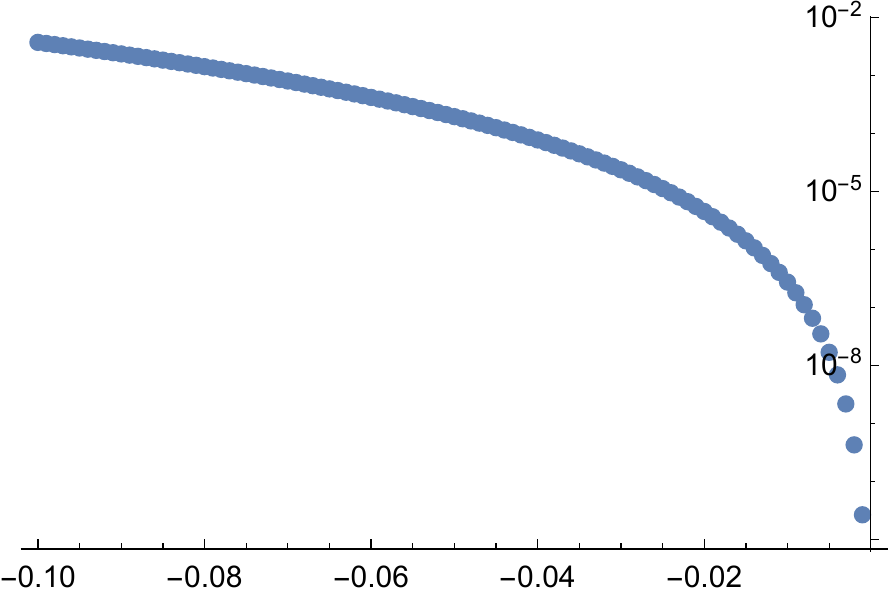}  \\
		\includegraphics[width=0.2\linewidth]{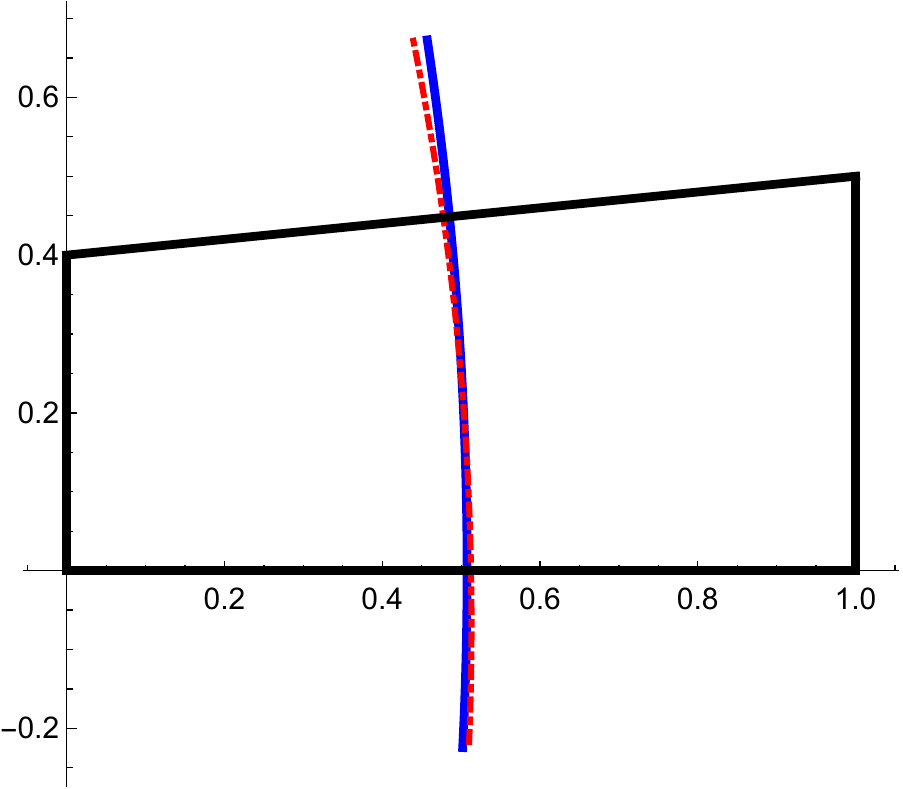}
		\caption{(Top Left) Ratio cut values corresponding to the optimal circular cuts (solid blue line) and parabolically approximated cuts (dot-dashed red line) for domains with various $a_1$ values between $-0.1$ and $0$. (Top Right)  The absolute error between the approximate and optimal ratio cut values in the left plot.  (Bottom) The optimal cut (solid blue line) and parabolically approximated cut (dot-dashed red line) for the trapezoidal domain with $a_1 = -0.1$.}		\label{compplot2}
		\end{figure}
			
	\begin{figure}	
	%AWL direction
		\includegraphics[width=0.25\linewidth]{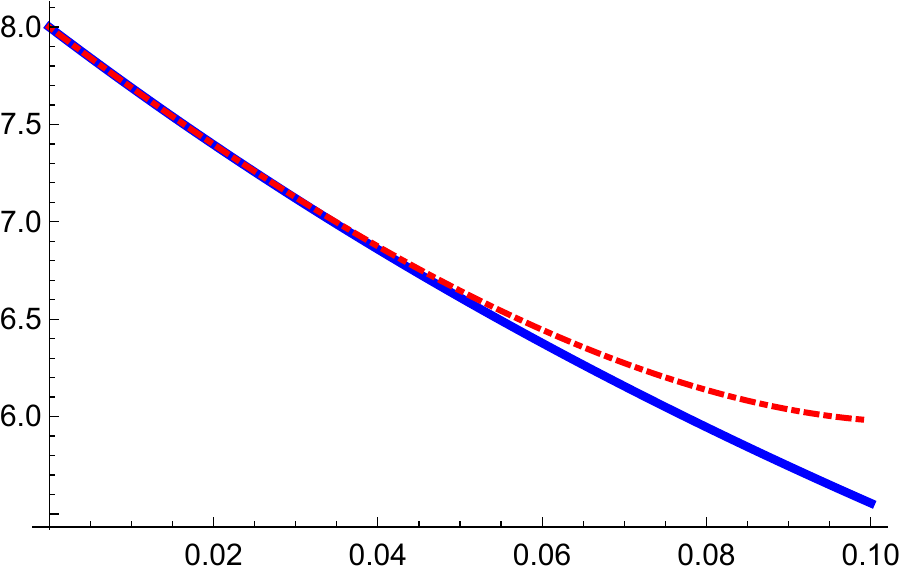}
		\includegraphics[width=0.25\linewidth]{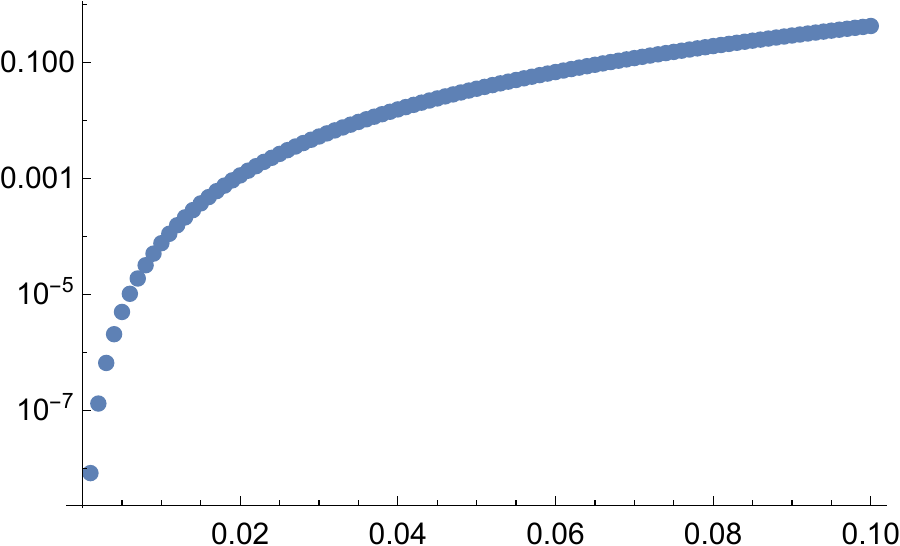}  \\
		\includegraphics[width=0.2\linewidth]{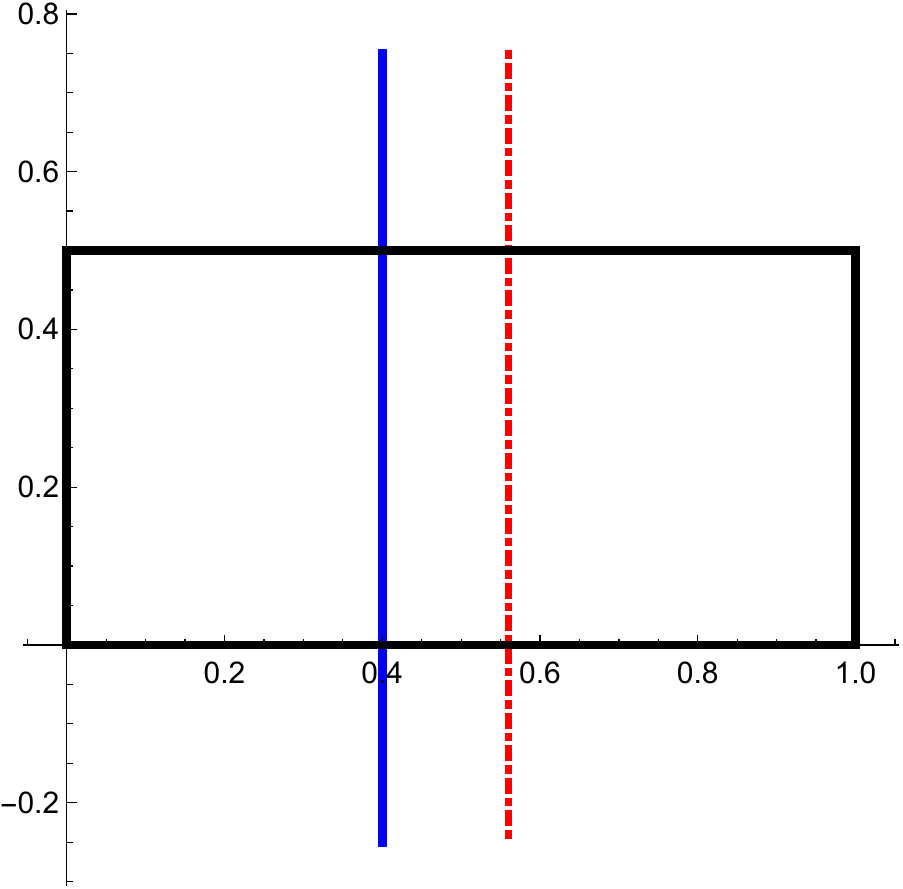}
		\caption{(Top Left) Ratio cut values corresponding to the optimal circular cuts (solid blue line) and parabolically approximated cuts (dot-dashed red line) for domains with various $A_{WL}$ values between $0$ and $0.1$. (Top Right)  The absolute error between the approximate and optimal ratio cut values in the left plot.  (Bottom) The optimal cut (solid blue line) and parabolically approximated cut (dot-dashed red line) for the trapezoidal domain with $A_{WL} = 0.1$.}
				\label{compplot3}
		\end{figure}
	\begin{figure}
		\includegraphics[width=0.25\linewidth]{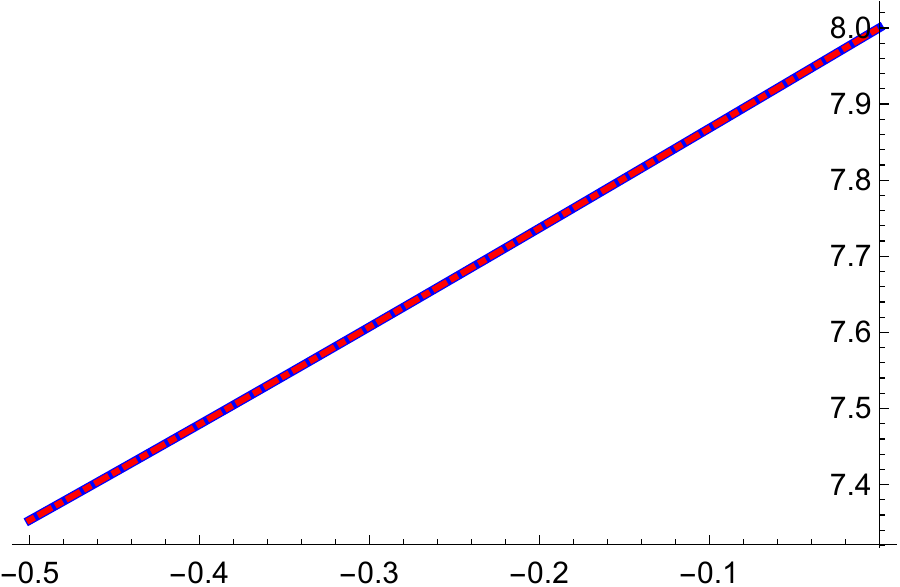}
		\includegraphics[width=0.25\linewidth]{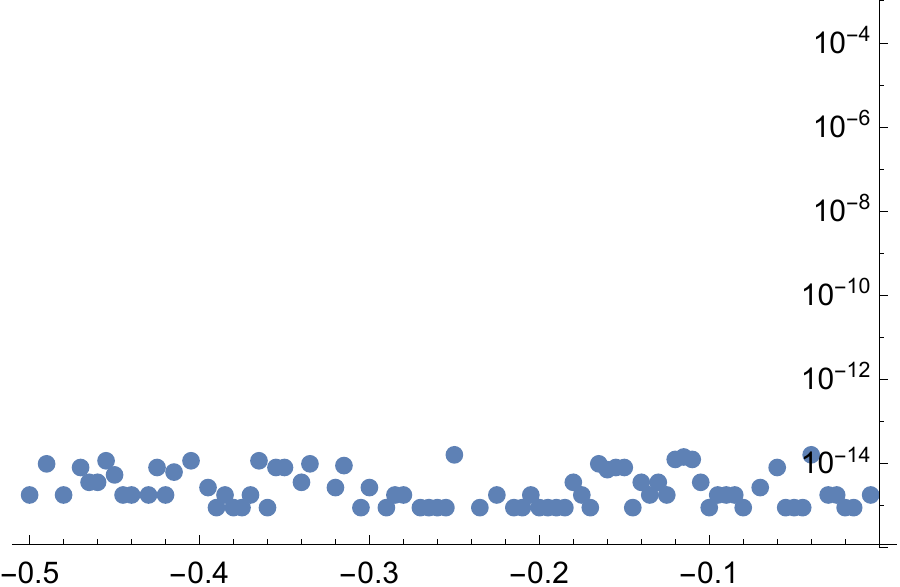} \\
		\includegraphics[width=0.2\linewidth]{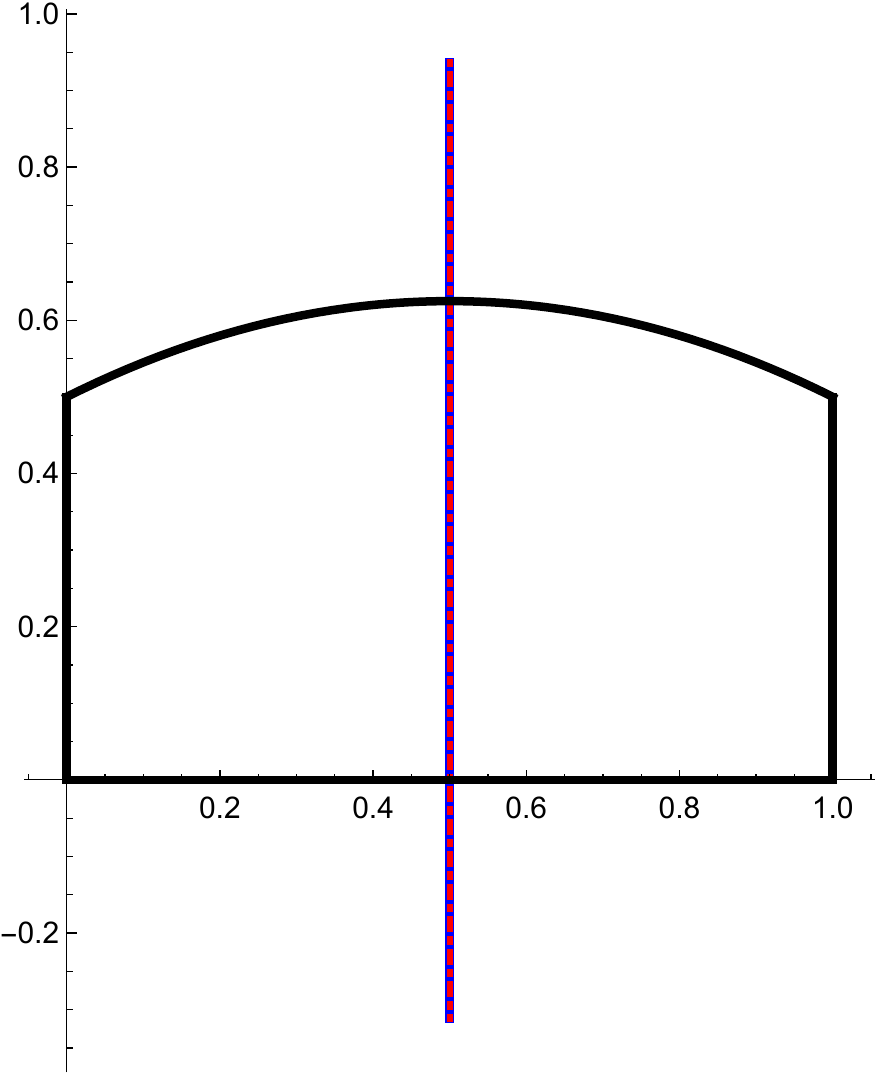}
		\caption{(Top Left) Ratio cut values corresponding to the optimal circular cuts (solid blue line) and parabolically approximated cuts (dot-dashed red line) for domains with various $\epsilon_t$ values between $-0.5$ and $0$. (Top Right)  The absolute error between the approximate and optimal ratio cut values in the left plot.  (Bottom) The optimal cut (solid blue line) and parabolically approximated cut (dot-dashed red line) for the trapezoidal domain with $\epsilon_t = -0.5$.}
				\label{compplot4}
	\end{figure}
	
	\begin{figure}
	%a1 a3 direction, together
		\includegraphics[width=0.25\linewidth]{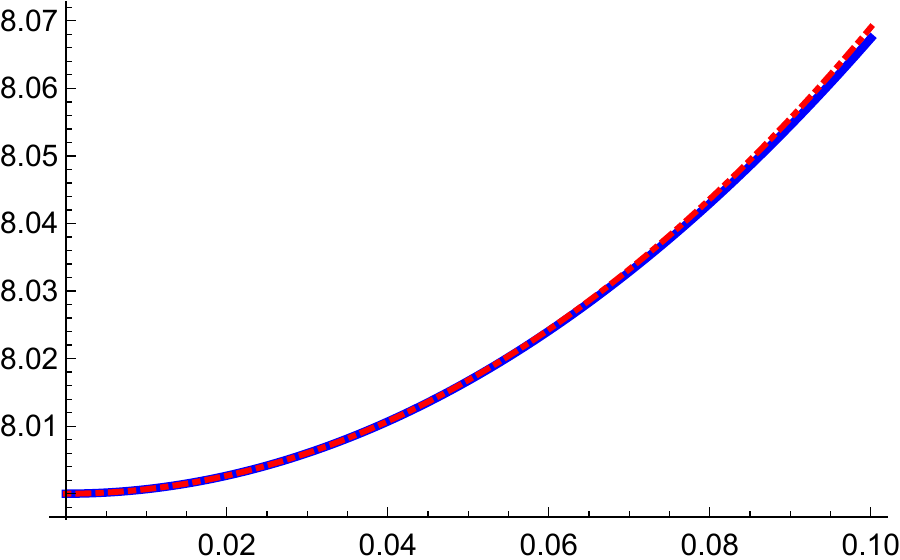}
		\includegraphics[width=0.25\linewidth]{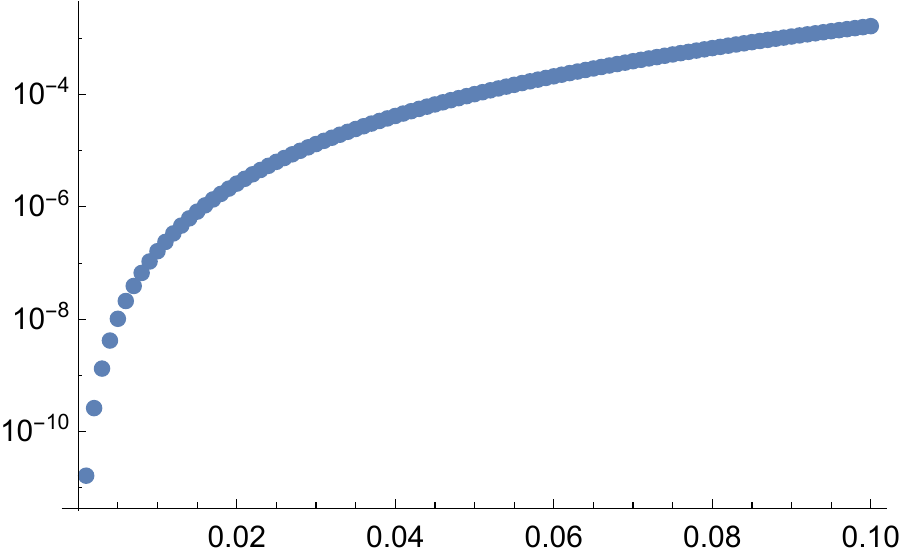} \\
		\includegraphics[width=0.2\linewidth]{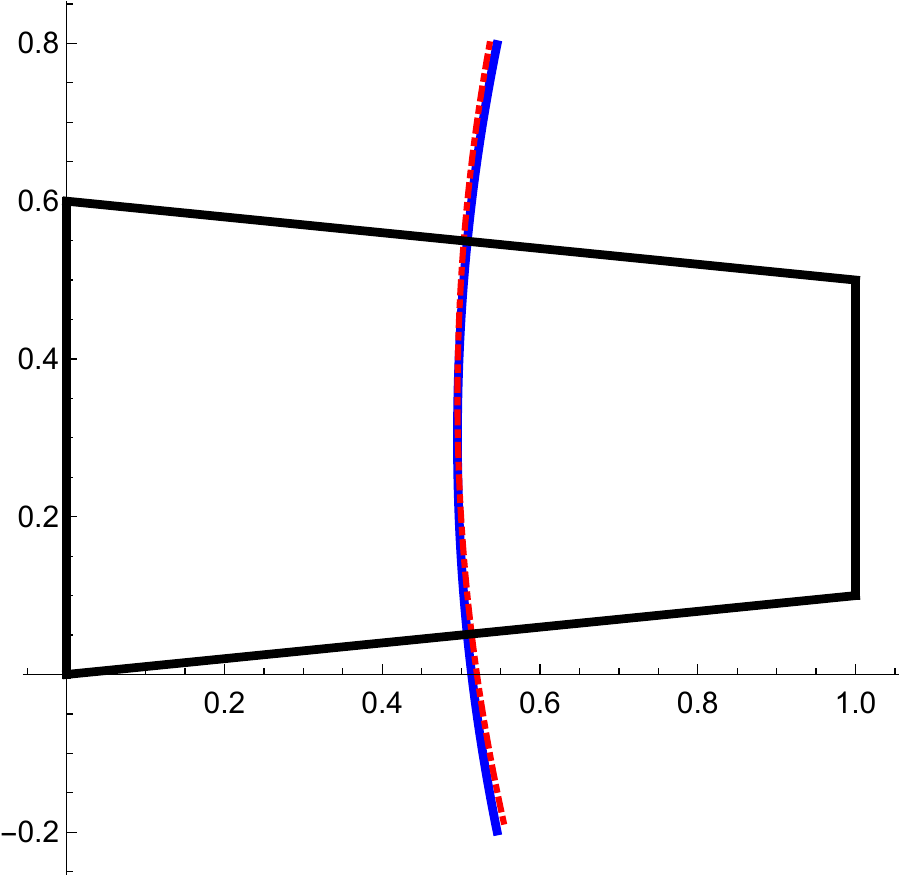}
		\caption{(Top Left) Ratio cut values corresponding to the optimal circular cuts (solid blue line) and parabolically approximated cuts (dot-dashed red line) for domains with various $a_1=a_3$ values between $0$ and $0.1$. (Top Right)  The absolute error between the approximate and optimal ratio cut values in the left plot.  (Bottom) The optimal cut (solid blue line) and parabolically approximated cut (dot-dashed red line) for the trapezoidal domain with $a_1 = a_3 = 0.1$.}
				\label{compplot5}
		\end{figure}
	\begin{figure}
		\includegraphics[width=0.25\linewidth]{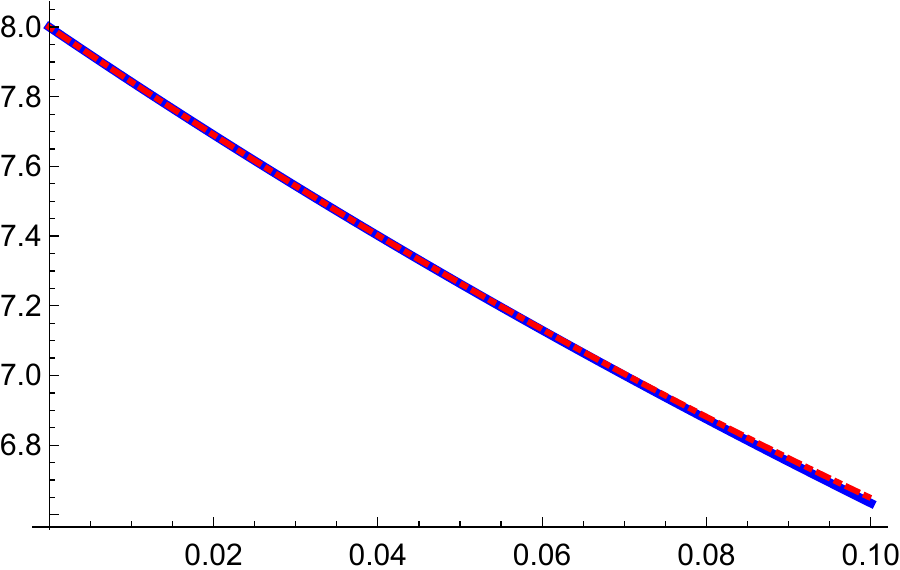}
		\includegraphics[width=0.25\linewidth]{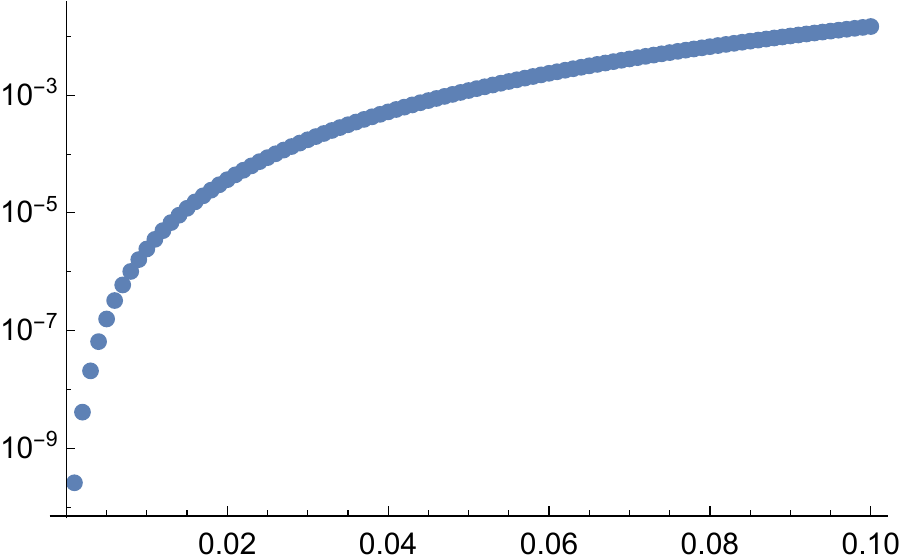} \\
		\includegraphics[width=0.2\linewidth]{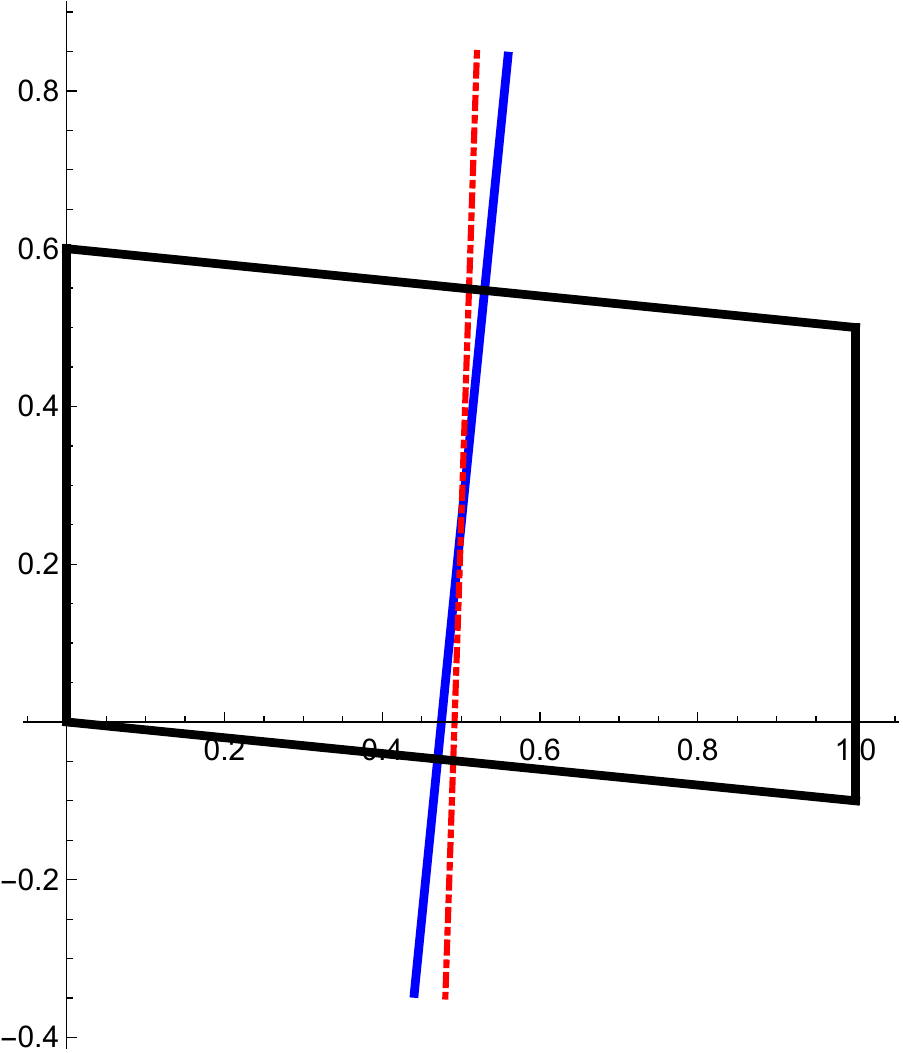}
		\caption{(Top Left) Ratio cut values corresponding to the optimal circular cuts (solid blue line) and parabolically approximated cuts (dot-dashed red line) for domains with various $a_1=-a_3$ values between $0$ and $0.1$. (Top Right)  The absolute error between the approximate and optimal ratio cut values in the left plot.  (Bottom) The optimal cut (solid blue line) and parabolically approximated cut (dot-dashed red line) for the trapezoidal domain with $a_1 = -a_3 = 0.1$.}
				\label{compplot6}
	\end{figure}
	
	\begin{figure}	
	%a1 a3 direction, together
		\includegraphics[width=0.25\linewidth]{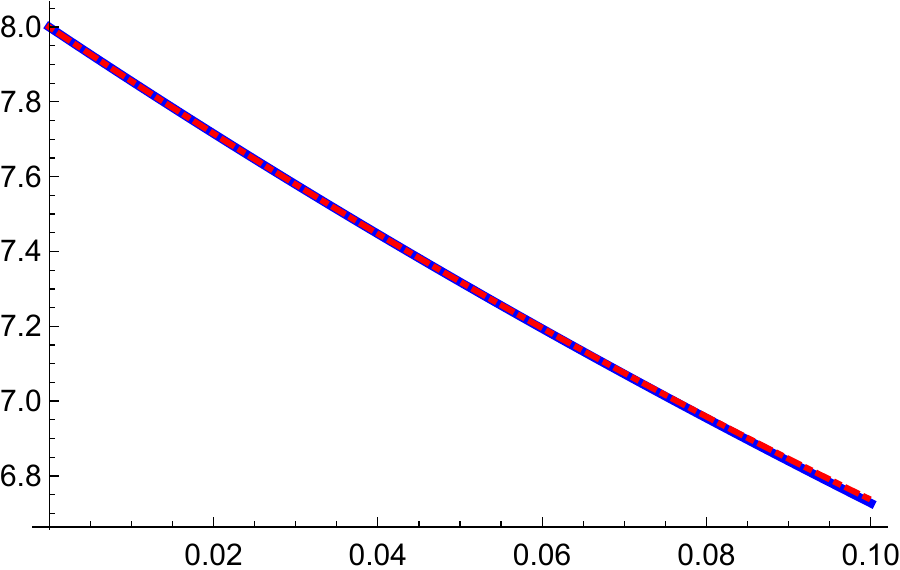}
		\includegraphics[width=0.25\linewidth]{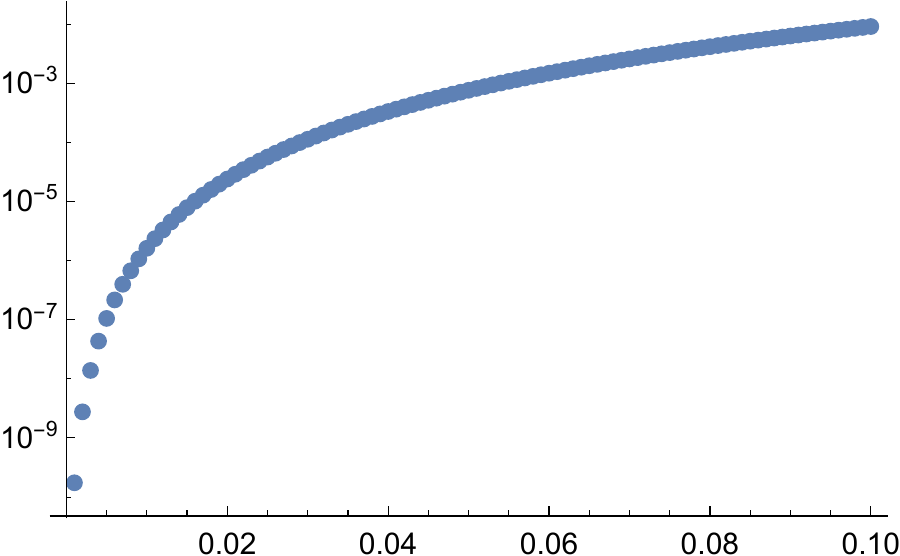} \\
		\includegraphics[width=0.2\linewidth]{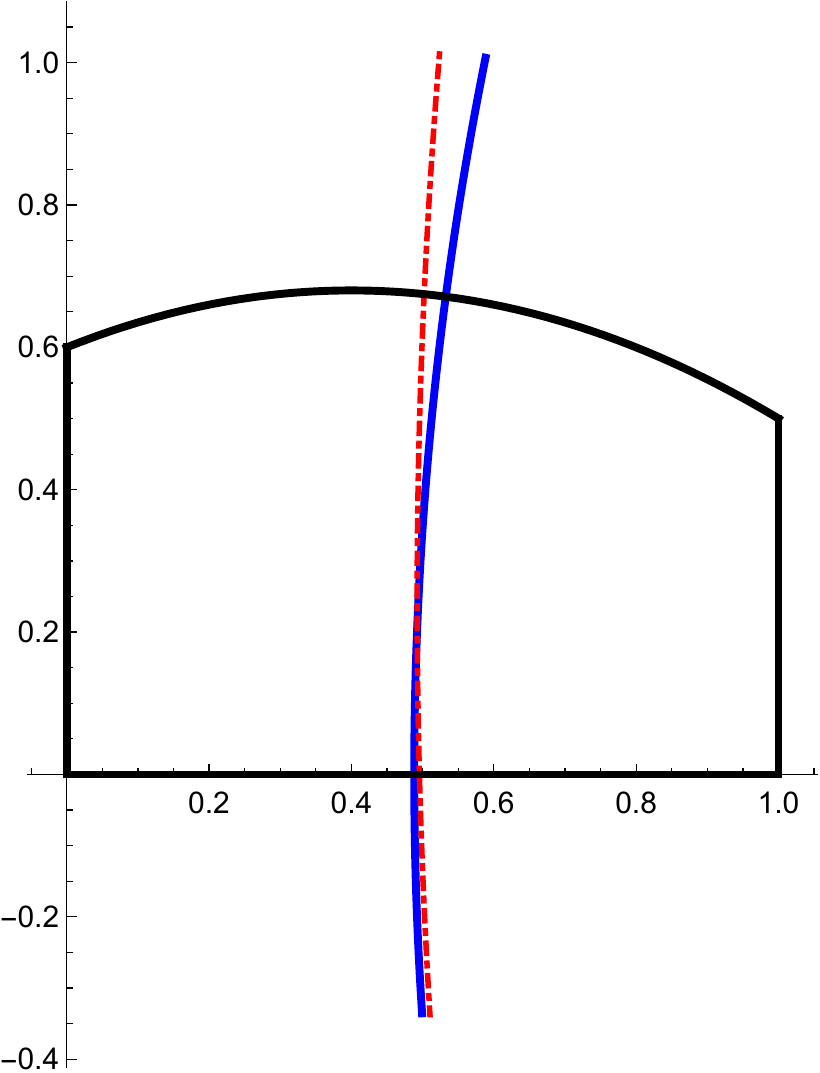}
		\caption{ (Top Left) Ratio cut values corresponding to the optimal circular cuts (solid blue line) and parabolically approximated cuts (dot-dashed red line) for domains with various $a_1=-\epsilon_t/5$ values between $0$ and $0.1$. (Top Right)  The absolute error between the approximate and optimal ratio cut values in the left plot.  (Bottom) The optimal cut (solid blue line) and parabolically approximated cut (dot-dashed red line) for the trapezoidal domain with $a_1 = -\epsilon_t/5 = 0.1$.}
				\label{compplot7}
		\end{figure}
	\begin{figure}
		\includegraphics[width=0.25\linewidth]{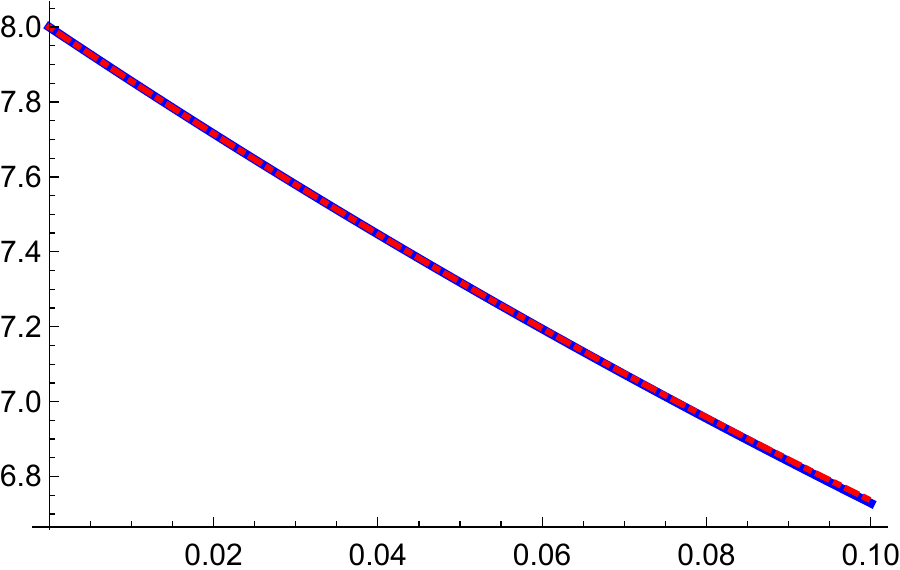}
		\includegraphics[width=0.25\linewidth]{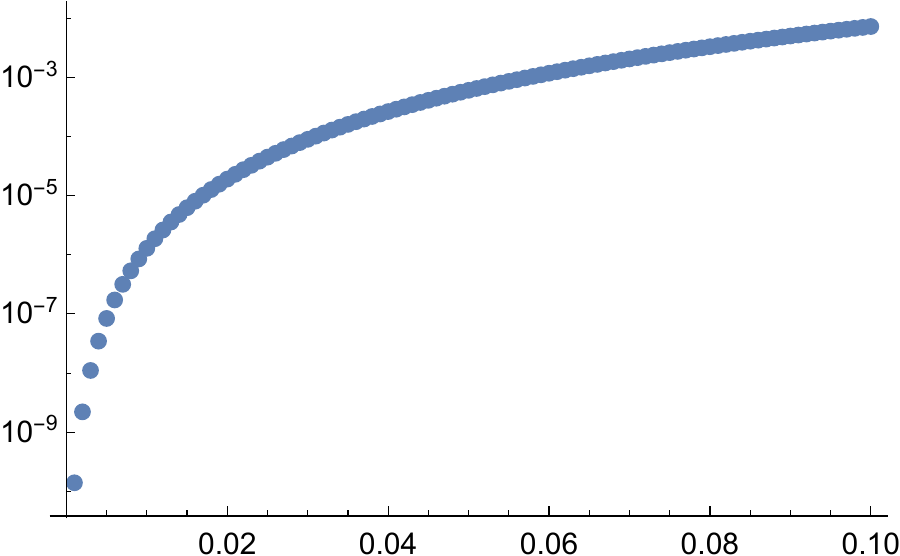}  \\
		\includegraphics[width=0.2\linewidth]{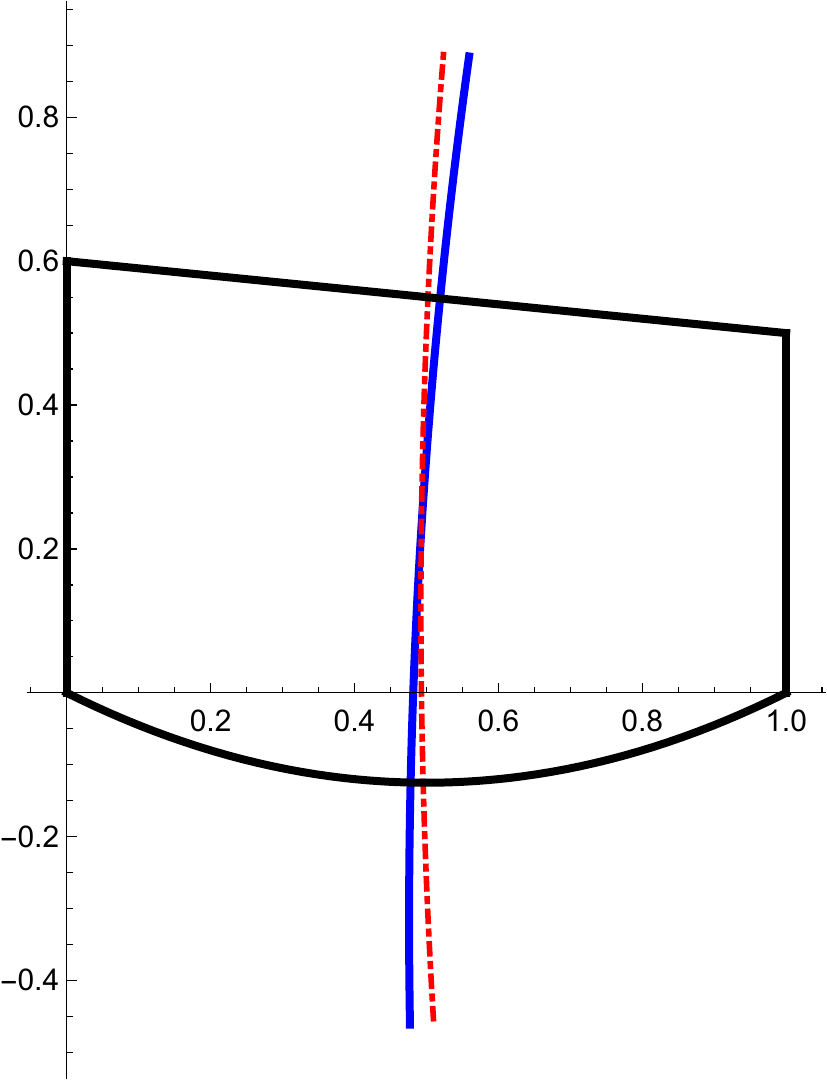}
		\caption{(Top Left) Ratio cut values corresponding to the optimal circular cuts (solid blue line) and parabolically approximated cuts (dot-dashed red line) for domains with various $a_1=\epsilon_b/5$ values between $0$ and $0.1$. (Top Right)  The absolute error between the approximate and optimal ratio cut values in the left plot.  (Bottom) The optimal cut (solid blue line) and parabolically approximated cut (dot-dashed red line) for the trapezoidal domain with $a_1 = \epsilon_b/5 = 0.1$.}
				\label{compplot8}
	\end{figure}

\clearpage
\appendix
\section{Taylor series coefficients}
	
	The polynomials in the Taylor series of the Ratio Cut are as follows:
	
	\noindent {\bf The $O(a_1)$ term:}
	\begin{align*}
	%p a_1
	p_{a_1} 	&= -8 + 8 \tBCent - 8 \tTCent + \frac{2}{3}\Th \\
	& \quad - 56 \tBCent^2 + 80 \tBCent \tTCent + 0 \tBCent \Th \\
	&\quad - 56 \tTCent^2 + 0 \tTCent \Th - \frac{5}{18} \Th^2,
		\end{align*}

	\noindent {\bf The $O(a_2)$ term:}		
	\begin{align*}
	%p a_2
	p_{a_2} 	&= -8 -8\tBCent + 8\tTCent - \frac{2}{3} \Th \\
	&\quad -56 \tBCent^2 +80 \tBCent \tTCent + 0 \tBCent \Th  \\
	&\quad -56 \tTCent^2 + 0 \tTCent \Th -\frac{5}{18} \Th^2, 
		\end{align*}
		
	\noindent {\bf The $O(a_3)$ term:}				
	\begin{align*}
	%p a_3
	p_{a_3} 	&= 8 - 8\tBCent + 8\tTCent + \frac{2}{3}\Th \\
	&\quad + 56 \tBCent^2 -80 \tBCent \tTCent + 0 \tBCent\Th \\
	&\quad + 56 \tTCent^2 + 0 \tTCent \Th + \frac{5}{18}\Th^2,
	\end{align*}
	
	\noindent {\bf The $O(\epsilon_t)$ term:}			
	\begin{align*}
	%p eps_t
	p_{\epsilon_t} 	&= \frac{4}{3} + 0 \tBCent + 0 \tTCent + 0\Th\\
	&\quad + \frac{52}{3}\tBCent^2 -40 \tBCent \tTCent -\frac{8}{9} \tBCent \Th  \\
	&\quad +\frac{100}{3} \tTCent^2 - \frac{8}{9} \tTCent \Th - \frac{1}{108} \Th^2 ,	
	\end{align*}
	
\noindent {\bf The $O(\epsilon_b)$ term:}			
	\begin{align*}
	%p eps_b
	p_{\epsilon_b} 	&= -\frac{4}{3} + 0\tBCent + 0\tTCent + 0\Th\\
	&\quad - \frac{100}{3} \tBCent^2 + 40 \tBCent \tTCent + \frac{8}{9} \tBCent \Th\\
	&\quad - \frac{52}{3} \tTCent^2 + \frac{8}{9} \tTCent \Th + \frac{1}{108} \Th^2,
	\end{align*}

\noindent {\bf The $O(A_{\texttt{WL}})$ term:}		
	\begin{align*}
	%p AWingL
	p_{A_{\texttt{WL}}} 	&= -32 + 32 \tBCent + 32 \tTCent + \frac{8}{3} \Th \\
	&\quad - 128 \tBCent^2 + 0 \tBCent \tTCent - \frac{32}{3} \tBCent \Th \\
	&\quad - 128 \tTCent^2 - \frac{32}{3}\tTCent \Th - \frac{16}{9} \Th^2  ,
		\end{align*}
		
		\noindent {\bf The $O(A_{\texttt{WR}})$ term:}		
	\begin{align*}
	%p AWingR
	p_{A_{\texttt{WR}}} 	&= -32  -32 \tBCent -32 \tTCent - \frac{8}{3} \Th\\
	&\quad -128 \tBCent^2 + 0 \tBCent\tTCent -\frac{32}{3} \tBCent\Th\\
	&\quad -128 \tTCent^2 - \frac{32}{3} \tTCent \Th - \frac{16}{9} \Th^2, 
	\end{align*}
	
	\noindent {\bf The $O(a_1^2)$ term:}		
	\begin{align*}
	%p a_1a_1
	p_{a_1a_1} 	&= 20 -32 \tBCent + 32 \tTCent - \frac{4}{3} \Th\\
	&\quad + 240 \tBCent^2 - 384 \tBCent\tTCent +4 \tBCent\Th\\
	&\quad + 208 \tTCent^2 -4  \tTCent \Th + \Th^2,
		\end{align*}
		
			\noindent {\bf The $O(a_1 a_2)$ term:}		
	\begin{align*}
	%p a_1a_2
	p_{a_1a_2} 	&= 12 + 0\tBCent + 0\tTCent + 0\Th\\
	&\quad + 176 \tBCent^2 - 320 \tBCent\tTCent -4 \tBCent\Th\\
	&\quad + 208 \tTCent^2 + 4\tTCent \Th + \frac{1}{3} \Th^2,
		\end{align*}
		
			\noindent {\bf The $O(a_1 a_3)$ term:}		
	\begin{align*}
	%p a_1a_3
	p_{a_1a_3} 	&= -12 + 32 \tBCent -32 \tTCent + 0 \Th\\
	&\quad -192 \tBCent^2 + 320 \tBCent\tTCent + 0\tBCent\Th\\
	&\quad -192  \tTCent^2 + 0\tTCent \Th - \frac{1}{3} \Th^2, 
	\end{align*}
	
		\noindent {\bf The $O(a_2^2)$ term:}		
	\begin{align*}
	%p a_2a_2
	p_{a_2a_2} 	&= 20 + 32 \tBCent -32  \tTCent + \frac{4}{3} \Th\\
	&\quad + 240 \tBCent^2 -384 \tBCent\tTCent + 4 \tBCent\Th\\
	&\quad + 208 \tTCent^2 - 4 \tTCent \Th + \Th^2,
		\end{align*}
		
		\noindent {\bf The $O(a_2 a_3)$ term:}			
	\begin{align*}
	%p a_2a_3
	p_{a_2a_3} 	&= -20 + 0 \tBCent + 0 \tTCent - \frac{4}{3} \Th\\
	&\quad -224 \tBCent^2 + 384 \tBCent\tTCent + 0 \tBCent\Th\\
	&\quad -224 \tTCent^2 + 0 \tTCent \Th - \Th^2,
		\end{align*}
		
			\noindent {\bf The $O(a_3^2)$ term:}		
	\begin{align*}
	%p a_3a_3
	p_{a_3a_3} 	&= 20 - 32 \tBCent + 32 \tTCent + \frac{4}{3} \Th\\
	&\quad + 208 \tBCent^2 -384 \tBCent\tTCent - 4 \tBCent\Th\\
	&\quad + 240 \tTCent^2 + 4 \tTCent \Th + \Th^2,
	\end{align*}
	
		\noindent {\bf The $O(a_1 A_{\texttt{WL}} )$ term:}		
	\begin{align*}
	%p a_1 A_WL
	p_{a_1 A_{\texttt{WL}}} 	&= 80 - 128 \tBCent -64 \tTCent - 8 \Th\\
	&\quad + 512 \tBCent^2 - 256 \tBCent\tTCent + \frac{80}{3} \tBCent\Th\\
	&\quad + 448 \tTCent^2 + \frac{32}{3} \tTCent \Th + 4 \Th^2,
		\end{align*}
	
		\noindent {\bf The $O(a_1 A_{\texttt{WR}} )$ term:}				
	\begin{align*}
	%p a_1 A_WR
	p_{a_1 A_{\texttt{WR}}} 	&= 48 + 0 \tBCent + 64 \tTCent - \frac{8}{3} \Th\\
	&\quad + 256 \tBCent^2  - 256  \tBCent\tTCent - \frac{16}{3} \tBCent\Th\\
	&\quad + 320 \tTCent^2 + \frac{32}{3} \tTCent \Th + \frac{4}{3} \Th^2,
		\end{align*}
		
		\noindent {\bf The $O(a_1 \epsilon_t )$ term:}				
	\begin{align*}
	%p a_1 epsT
	p_{a_1 \epsilon_t} &= -\frac{8}{3} + \frac{8}{3} \tBCent - 8 \tTCent -\frac{1}{9} \Th\\
	&\quad -72 \tBCent^2 + \frac{464}{3} \tBCent\tTCent + \frac{8}{9} \tBCent\Th\\
	&\quad - 104 \tTCent^2 + \frac{8}{9} \tTCent \Th - \frac{1}{9} \Th^2,
		\end{align*}
		
				\noindent {\bf The $O(a_1 \epsilon_b )$ term:}				
	\begin{align*}
	%p a_1 epsB
	p_{a_1 \epsilon_b} &= \frac{8}{3} -\frac{8}{3} \tBCent + 8 \tTCent + \frac{1}{9} \Th\\
	&\quad + 104 \tBCent^2 - \frac{464}{3}  \tBCent\tTCent - \frac{8}{9} \tBCent\Th\\
	&\quad + 72 \tTCent^2 - \frac{8}{9} \tTCent \Th + \frac{1}{9} \Th^2,
	\end{align*} 
	
			\noindent {\bf The $O(a_2 A_{\texttt{WL}}  )$ term:}				
	\begin{align*}
	%p a_2 A_WL
	p_{a_2 A_{\texttt{WL}}} 	&= 48 + 0 \tBCent - 64 \tTCent + \frac{8}{3} \Th\\
	&\quad + 256 \tBCent^2 - 256 \tBCent\tTCent + - \frac{16}{3}\tBCent\Th\\
	&\quad + 320 \tTCent^2 + \frac{32}{3} \tTCent \Th + \frac{4}{3} \Th^2,
		\end{align*}
		
					\noindent {\bf The $O(a_2 A_{\texttt{WR} } )$ term:}				
	\begin{align*}
	%p a_2 A_WR
	p_{a_2 A_{\texttt{WR}}} 	&= 80 + 128  \tBCent + 64 \tTCent + 8 \Th\\
	&\quad + 512 \tBCent^2 + \frac{80}{3} \tBCent\tTCent + \tBCent\Th\\
	&\quad + 448 \tTCent^2 + \frac{32}{3} \tTCent \Th + 4 \Th^2,
		\end{align*}
		
					\noindent {\bf The $O(a_2 \epsilon_t )$ term:}				
	\begin{align*}
	%p a_2 epsT
	p_{a_2 \epsilon_t} &= - \frac{8}{3} - \frac{8}{3} \tBCent + 8 \tTCent + \frac{1}{9} \Th\\
	&\quad - 72 \tBCent^2 + \frac{464}{3} \tBCent\tTCent + \frac{8}{9} \tBCent\Th\\
	&\quad - 104 \tTCent^2 + \frac{8}{9} \tTCent \Th - \frac{1}{9} \Th^2	,
	\end{align*}
	
						\noindent {\bf The $O(a_2 \epsilon_b )$ term:}				
	\begin{align*}
	%p a_2 epsB
	p_{a_2 \epsilon_b} &= \frac{8}{3} + \frac{8}{3} \tBCent - 8 \tTCent - \frac{1}{9} \Th\\
	&\quad + 104 \tBCent^2 - \frac{464}{3} \tBCent\tTCent - \frac{8}{9} \tBCent\Th\\
	&\quad + 72 \tTCent^2 - \frac{8}{9} \tTCent \Th + \frac{1}{9} \Th^2,
	\end{align*}
	
						\noindent {\bf The $O(a_3 A_{\texttt{WL}} )$ term:}				
	\begin{align*}
	%p a_3 A_WL
	p_{a_3 A_{\texttt{WL}}} 	&= -48 + 64 \tBCent + 0 \tTCent - \frac{8}{3} \Th\\
	&\quad -320 \tBCent^2 + 256 \tBCent\tTCent - \frac{32}{3} \tBCent\Th\\
	&\quad - 256 \tTCent^2 + \frac{16}{3} \tTCent \Th - \frac{4}{3} \Th^2,
		\end{align*}

\noindent {\bf The $O(a_3 A_{\texttt{WR} })$ term:}	
	\begin{align*}
	%p a_3 A_WR
	p_{a_3 A_{\texttt{WR}}} 	&= -80 -64 \tBCent -128 \tTCent - 8 \Th\\
	&\quad -448 \tBCent^2 + 256 \tBCent\tTCent - \frac{32}{3} \tBCent\Th\\
	&\quad -512 \tTCent^2 - \frac{80}{3} \tTCent \Th - 4 \Th^2,
		\end{align*}
		
\noindent {\bf The $O(a_3 \epsilon_t )$ term:}	
	\begin{align*}
	%p a_3 A_epsT
	p_{a_3 \epsilon_t} &= \frac{8}{3} - 8 \tBCent + \frac{8}{3} \tTCent - \frac{1}{9} \Th\\
	&\quad + 72 \tBCent^2 - \frac{464}{3} \tBCent\tTCent - \frac{8}{9} \tBCent\Th\\
	&\quad + 104 \tTCent^2 - \frac{8}{9} \tTCent \Th + \frac{1}{9} \Th^2,
		\end{align*}
		
		\noindent {\bf The $O(a_3 \epsilon_b)$ term:}	
	\begin{align*}
	%p a_3 epsB
	p_{a_3 \epsilon_b} &= -\frac{8}{3} + 8 \tBCent - \frac{8}{3} \tTCent + \frac{1}{9} \Th\\
	&\quad - 104 \tBCent^2 + \frac{464}{3} \tBCent\tTCent + \frac{8}{9} \tBCent\Th\\
	&\quad - 72 \tTCent^2 + \frac{8}{9} \tTCent \Th - \frac{1}{9} \Th^2,
	\end{align*}
	
	\noindent {\bf The $O(A_{\texttt{WL}}^2)$ term:}	
	\begin{align*}
	%p A_WL A_WL
	p_{A_{\texttt{WL}}A_{\texttt{WL}}} 	&= 256 - 512 \tBCent -512 \tTCent - \frac{128}{3} \Th\\
	&\quad + 1536 \tBCent^2 + 1024 \tBCent\tTCent + \frac{512}{3} \tBCent\Th\\
	&\quad + 1536 \tTCent^2 + \frac{512}{3} \tTCent \Th + \frac{160}{9} \Th^2,
		\end{align*}
		
	\noindent {\bf The $O(A_{\texttt{WL}}  A_{\texttt{WR}})$ term:}			
	\begin{align*}
	%p A_WL A_WR
	p_{A_{\texttt{WL}}A_{\texttt{WR}}}	&= 128 + 0 \tBCent + 0 \tTCent + 0 \Th\\
	&\quad + 512 \tBCent^2 + 0 \tBCent\tTCent + \frac{128}{3} \tBCent\Th\\
	&\quad + 512 \tTCent^2 + \frac{128}{3} \tTCent \Th + \frac{64}{9} \Th^2,
		\end{align*}
		
			\noindent {\bf The $O(A_{\texttt{WR}}^2)$ term:}	
	\begin{align*}
	%p A_WR A_WR
	p_{A_{\texttt{WR}} A_{\texttt{WR}}}	&= 256 + 512 \tBCent + 512  \tTCent + \frac{128}{3} \Th\\
	&\quad + 1536 \tBCent^2 + 1024 \tBCent\tTCent + \frac{512}{3} \tBCent\Th\\
	&\quad + 1536 \tTCent^2 + \frac{512}{3} \tTCent \Th + \frac{160}{9} \Th^2,
	\end{align*}
	
		\noindent {\bf The $O(A_{\texttt{WL}}  \epsilon_t)$ term:}	
	\begin{align*}
	%p A_WL epsT
	p_{A_{\texttt{WL}} \epsilon_t} 	&= -16 + \frac{32}{3} \tBCent + \frac{32}{3} \tTCent - \frac{4}{9} \Th\\
	&\quad - \frac{320}{3} \tBCent^2 + \frac{512}{3} \tBCent\tTCent + \frac{32}{9} \tBCent\Th\\
	&\quad - \frac{512}{3} \tTCent^2 + \frac{32}{9} \tTCent \Th - \frac{8}{27} \Th^2,
		\end{align*}
		
		\noindent {\bf The $O(A_{\texttt{WL}}  \epsilon_b)$ term:}			
	\begin{align*}
	%p A_WL epsB
	p_{A_{\texttt{WL}} \epsilon_b} 	&= 16 - \frac{32}{3} \tBCent - \frac{32}{3} \tTCent + \frac{4}{9} \Th\\
	&\quad + \frac{512}{3} \tBCent^2 - \frac{512}{3} \tBCent\tTCent - \frac{32}{9} \tBCent\Th\\
	&\quad + \frac{320}{3} \tTCent^2 - \frac{32}{9} \tTCent \Th + \frac{8}{27} \Th^2,
		\end{align*}
		
		\noindent {\bf The $O(A_{\texttt{WR}}  \epsilon_t)$ term:}	
	\begin{align*}
	%p A_WR epsT
	p_{A_{\texttt{WR}} \epsilon_t} 	&= -16 - \frac{32}{3} \tBCent - \frac{32}{3} \tTCent + \frac{4}{9} \Th\\
	&\quad - \frac{320}{3} \tBCent^2 + \frac{512}{3} \tBCent\tTCent + \frac{32}{9} \tBCent\Th\\
	&\quad - \frac{512}{3} \tTCent^2 + \frac{32}{9} \tTCent \Th - \frac{8}{27} \Th^2,
		\end{align*}
		
			\noindent {\bf The $O(A_{\texttt{WR}}  \epsilon_b)$ term:}	
	\begin{align*}
	%p A_WR epsB
	p_{A_{\texttt{WR}} \epsilon_b} 	&= 16 + \frac{32}{3} \tBCent + \frac{32}{3} \tTCent - \frac{4}{9} \Th\\
	&\quad + \frac{512}{3} \tBCent^2 - \frac{512}{3} \tBCent\tTCent  - \frac{32}{3} \tBCent\Th\\
	&\quad + \frac{320}{3} \tTCent^2 - \frac{32}{9} \tTCent \Th + \frac{8}{27} \Th^2,
	\end{align*}
	
			\noindent {\bf The $O( \epsilon_t^2)$ term:}	
	\begin{align*}
	%p epsT epsT
	p_{\epsilon_t \epsilon_t} 	&= 0 + 0 \tBCent + 0 \tTCent + 0 \Th\\
	&\quad + \frac{244}{9} \tBCent^2 - \frac{568}{9}  \tBCent\tTCent - \frac{4}{27} \tBCent\Th\\
	&\quad + \frac{436}{9} \tTCent^2 + \frac{4}{27} \tTCent \Th + \frac{5}{162} \Th^2,
		\end{align*}
		
					\noindent {\bf The $O( \epsilon_t  \epsilon_b)$ term:}	
	\begin{align*}
	%p epsT epsB
	p_{\epsilon_t \epsilon_b} 	&= 0 + 0 \tBCent + 0 \tTCent + 0 \Th\\
	&\quad - \frac{340}{9} \tBCent^2 + \frac{568}{9} \tBCent\tTCent + \frac{4}{27} \tBCent\Th\\
	&\quad - \frac{340}{9} \tTCent^2 + \frac{4}{27} \tTCent \Th - \frac{5}{162} \Th^2,
		\end{align*}
	
		\noindent {\bf The $O(\epsilon_b^2)$ term:}	
	\begin{align*}
	%p epsB epsB
	p_{\epsilon_b \epsilon_b} 	&= 0 + 0 \tBCent + 0 \tTCent + 0 \Th\\
	&\quad + \frac{436}{9} \tBCent^2 - \frac{568}{9} \tBCent\tTCent - \frac{4}{27} \tBCent\Th\\
	&\quad + \frac{244}{9} \tTCent^2 + \frac{4}{27} \tTCent \Th + \frac{5}{162} \Th^2.
	\end{align*}

\end{document}